\documentclass[10pt]{article}
   \usepackage{graphicx}
   \usepackage{epsfig}
   \usepackage{amssymb}
   \usepackage{amsmath}
   \usepackage{amsthm}
   \newtheorem{lemma}{Lemma}[section]
   \newtheorem{theorem}{Theorem}[section]

   {\end{description}}

   {\hfill $\bullet$ \\}

   \newcommand{\be}{\begin{equation}}
   \newcommand{\ee}{\end{equation}}

\begin{document}
    \title{A strong two-stage explicit/implicit approach combined with mixed finite element methods for a three-dimensional nonlinear radiation-conduction model in anisotropic media}
   \author{Eric Ngondiep}
   \date{$^{\text{\,1\,}}$\small{Department of Mathematics and Statistics, College of Science, Imam Mohammad Ibn Saud\\ Islamic University
        (IMSIU), $90950$ Riyadh $11632,$ Saudi Arabia.}\\
     \text{\,}\\
       $^{\text{\,2\,}}$\small{Research Center for Water and Climate Change, Institute of Geological and Mining Research, 4110 Yaounde-Cameroon.}\\
     \text{\,}\\
        \textbf{Email addresses:} ericngondiep@gmail.com/engondiep@imamu.edu.sa}

   \maketitle

   \textbf{Abstract.}
   This paper develops a strong computational approach to simulate a three-dimensional nonlinear radiation-conduction model in optically thick media, subject to suitable initial and boundary conditions. The space derivatives are approximated by the mixed finite element method ($\mathcal{P}_{p}/\mathcal{P}_{p-1}$), while the interpolation technique is employed in two stages to approximate the time derivative. The proposed strategy is so-called, a strong two-stage explicit/implicit computational technique combined with mixed finite element method. Specifically, the new algorithm should be observed as a predictor-corrector numerical scheme. Additionally, it efficiently treats the time derivative term and provides a necessary requirement on time step for stability. Under this time step limitation, the stability is deeply analyzed whereas the convergence order is numerically computed in the $L^{2}$-norm. The theoretical results suggest that the developed approach is stable and temporal second-order accurate. Some numerical experiments are performed to confirm the theory, to establish that the constructed method is spatial fourth-order convergent and to demonstrate the practical applicability and computational efficiency of the numerical scheme.\\
    \text{\,}\\

   \ \noindent {\bf Keywords:} $3D$ radiation-conduction model in anisotropic media, simplified $SP_{3}$ approximations, two-stage explicit/implicit computational technique, mixed finite element formulation ($\mathcal{P}_{p}/\mathcal{P}_{p-1}$), stability analysis.\\
   \\
   {\bf AMS Subject Classification (MSC): 65M12, 65M06}.

  \section{Introduction}\label{sec1}

   \text{\,\,\,\,\,\,\,\,\,\,}Modeling and simulating radiative heat transfer is crucial in various high-temperature applications, including strong industrial furnaces \cite{13bs,16bs}, gas turbine combustion chambers \cite{34bs}, continuous casting \cite{8bs,32bs}, and glass manufacturing \cite{14bs}. The key processes used in simulating heat transport in these applications are conduction and radiation mechanisms, as well as convection and chemical reactions. In contrast to radiation, which transmits heat at the speed of light and can affect objects far away from the heat source, conduction occurs over a longer time scale and transports heat near the heat source. Radiative transport is especially important in these applications at high temperatures and should not be overlooked when modeling them. For example, research on semi-transparent materials has shown that radiation and conduction are not the exclusive methods for measuring heat transport. In fact, the media temperature in many annealing methods exceeds $1000 K$, and conduction may be dominated by radiation. In these applications, radiative heat transfer must be mathematically described by solving integro-differential equations for radiative transfer coupled with a class of parabolic partial differential equations (PDEs) for the medium's thermal behavior (for example, see \cite{29bs}). However, because of the wide set of dependent unknowns, the coupled radiation-conduction model, and specular reflecting boundary conditions, these equations fall in the class of complex time-dependent PDEs \cite{1en,2en,3en,6en,4en,5en,7en} whose the computation of analytical solutions is too difficult and sometimes impossible.\\

   In this paper, we consider the following three-dimensional nonlinear radiation-conduction model in optically thick media defined in \cite{bs} as
   \begin{equation}\label{1}
    \begin{array}{c}
    \partial_{t}T-\nabla\cdot(M\nabla T)=\frac{k_{0}}{\tau^{2}}\nabla\left(-4\pi f(T)+\int_{\mathcal{S}}I(x,s)ds\right), \text{\,\,\,\,on\,\,\,}
     \Omega\times[0,\text{\,}T_{f}]\\
     \text{\,}\\
       \tau s^{t}\nabla I+(k_{0}+\sigma_{0})I=kf(T)+\frac{\sigma_{0}}{4\pi}\int_{\mathcal{S}}I(x,s)ds, \text{\,\,\,\,\,}(x,s)\in\Omega\times \mathcal{S}\\
    \end{array}
   \end{equation}
   subject to initial condition
   \begin{equation}\label{2}
    T(x,0)=T_{0}(x),  \text{\,\,\,\,for\,\,\,}x\in\Omega\cup\Gamma
   \end{equation}
   and boundary conditions
   \begin{equation}\label{3}
    \begin{array}{c}
     \tau\vec{n}(x)^{t}M\nabla T+c_{m}(T-T_{m})=-\alpha\pi(f(T)-f(T_{m})), \text{\,\,\,\,on\,\,\,}\Gamma\times[0,\text{\,}T_{f}]\\
     \text{\,}\\
      I(x,s)=f(T_{m}), \text{\,\,\,\,\,}(x,s)\in\Gamma^{-}\times\mathcal{S}\\
    \end{array}
   \end{equation}
   where $\partial_{t}T$ denotes $\frac{\partial T}{\partial t}$, $\Omega$ is a bounded domain of an absorbing and emitting material in $\mathbb{R}^{3}$, $\Gamma$
   represents its boundary and $T_{f}$ is the final time. $k_{0}$, $c_{m}$, $\sigma_{0}$, $\tau$, $\alpha$, $T_{m}$ and $M$ designate the absorption coefficient, convective
   heat transfer coefficient, scattering coefficient, diffusion scale, hemispheric surface emissivity, ambient temperature of the surrounding and thermal conductivity, respectively. $T$ and $I$ are the unknown temperature and radiative intensity, respectively. $\mathcal{S}$ denotes the unit sphere, $\vec{n}(x)$ is the outward normal vector at point $x$ to the boundary $\Gamma$, $u^{t}$ is the transpose of a vector $u$, and $f(T)$ represents the spectral intensity of the black-body radiation defined as
   \begin{equation}\label{3a}
    f(T)=c_{bs}T^{4},
   \end{equation}
   where $c_{bs}=5.67\times10^{-8}$, is the Stefan-Boltzmann constant given in \cite{29bs}. It's worth mentioning that the boundary $\Gamma^{-}$ is defined as:
   $\Gamma^{-}=\{x\in\Gamma,\text{\,\,}s^{t}\vec{n}(x)<0,\text{\,\,}s\in\mathcal{S}\}.$\\

    In the literature, several authors modeled radiative heat transport problems in participating media using the system of equations $(\ref{1})$. For more details, interested readers can consult \cite{29bs} and references therein. Numerical approaches such as discrete ordinates and zonal techniques have been developed to overcome these problems (\cite{42bs}). However, computational techniques for radiative heat transport in three dimensions require significant processing time and memory capacity. Large-scale numerical simulations necessitate efficient and stable numerical solvers to solve the radiation-conduction equations in anisotropic media. Over the last decade, researchers in radiative heat transfer have focused on developing simple approximations that accurately predict fundamental physical events at a reasonable computational cost. The simplified spherical harmonics technique (also known as $SP_{N}$ approximations) is the most widely used approximate model when the medium is isotropic and optically thick \cite{bs}. In \cite{36bs}, the results of radiative transmission in combustion systems were validated using $SP_{N}$ approximations and experimental data. It is worth mentioning that the primary advantage of the $SP_{N}$ approximations considers a set of elliptic-type PDEs independent of angular direction rather than integro-differential radiative transfer equations. Furthermore, some comparative results show that in  optically thick media, $SP_{N}$ models approximate the whole radiative heat transfer problem at a very low computational cost. High-temperature thermal engineering
    applications require anisotropy, as conductivity becomes nonlinear with temperature. Computational methods for solving $SP_{N}$ approximations have been widely discussed
    in the literature. Researchers have proposed a block Arnoldi method in \cite{43bs}, implemented continuous finite element methods in \cite{1bs}, investigated
    discontinuous finite element methods in \cite{21bs}, used enriched partition of unity finite element methods to solve $SP_{N}$ models and developed
    adaptive methods to accurately solve $SP_{N}$ models in \cite{9bs}. These strategies resulted in a significant increase in processing time for conduction-radiation problems requiring linear conductivity in isotropic materials. However, further study and development are needed to address nonlinear conduction-radiation difficulties in anisotropic media. All finite element methods (FEMs) for $SP_{N}$ problems use the unified approach, which use finite element spaces for both temperature and radiative solutions. This technique keeps the temperature solution at its highest precision while allowing the radiative solutions in the $SP_{N}$ approximations to lose accuracy. To overcome this limitation, we construct a two-stage explicit/implicit technique combined with the mixed finite element formulation ($\mathcal{P}_{p}/\mathcal{P}_{p-1}$), for solving the $SP_{3}$ model. The developed computational technique should guarantee efficient solutions for both temperature and related unknown functions.\\

  The new computational technique approximates the time derivative at the first stage (predictor step) using a second-order explicit scheme while an implicit second-order method is employed at the second stage (corrector step) so that the error increased at the predictor step is balanced by the one decreased at the corrector step. As a result, the stability of the constructed predictor-corrector scheme is maintained. Additionally, the mixed finite element formulation used to discretize the space derivatives considers two different spaces. Moreover, the mixed $\mathcal{P}_{p}/\mathcal{P}_{p-1}$ finite elements are applied with high-order elements $\mathcal{P}_{p}$ for the temperature solution and lower-order elements $\mathcal{P}_{p-1}$ for the radiative solutions. Thus, the nonlinear conduction-radiation problem in optically thick anisotropic media can be solved quickly and accurately. The development of the new algorithm deals with the three-dimensional $SP_{3}$ equations with nonlinear conductivity based on temperature, representing anisotropic radiative heat transport. Furthermore, some numerical examples for conduction-radiation problems, including a nonlinear problem in a three-dimensional anisotropic medium are carried out to show the efficiency of the proposed approach. The studies suggest that the proposed computational technique is faster and more efficient than a broad range of numerical methods discussed in the literature for solving coupled conduction-radiation problems \cite{bs,43bs}. It's worth recalling that the highlights of this study considers the following three items.\\

   \begin{description}
     \item[1.] Detailed description of the strong two-stage explicit/implicit computational technique combined with the mixed finite element formulation ($\mathcal{P}_{p}/\mathcal{P}_{p-1}$), for simulating the $SP_{3}$ approximations corresponding to the initial-boundary value problem $(\ref{1})$-$(\ref{3})$.
     \item[2.] Stability analysis of the proposed numerical approach under a necessary condition on the time step.
     \item[3.] Numerical examples to confirm the theory and validate the accuracy and efficiency of the new algorithm.
   \end{description}

   The remainder of the paper is organized as follows. Section $\ref{sec2}$ deals with the mathematical formulation of the nonlinear $SP_{3}$ approximations of radiation-conduction equations. In Section $\ref{sec3}$, we provide a detailed description of the strong two-stage explicit/implicit approach combined with the mixed finite element method in a computed solution of the nonlinear $SP_{3}$ approximations while the stability of the proposed numerical scheme is established under a necessary condition on the time step in Section $\ref{sec4}$. Finally, in Section $\ref{sec5}$, some numerical experiments are carried out to confirm the theory and to demonstrate the utility and performance of the new algorithm whereas the general conclusions and our future works are provided in Section $\ref{sec6}$.\\

     \section{Mathematical formulation of $SP_{3}$ approximations of nonlinear radiation-conduction equations}\label{sec2}

     \text{\,\,\,\,\,\,\,\,\,\,}In this section, we present a mathematical formulation of the nonlinear radiative heat transfer problem $(\ref{1})$ using the simplified spherical harmonics method ($SP_{3}$ approximations). These equations deal with several limitations when the numerical solutions are directly obtained from the differential equations governing the combustion, heat conduction and flow. The $SP_{3}$ problems are considered as an important approach to overcome these difficulties. In \cite{9bs,21bs}, the authors have shown that these simplified models are obtained utilizing an asymptotic analysis and they present an appropriate performance optically thick medium. In this paper, we construct a two-step explicit/implicit approach combined with a mixed finite element technique for solving the $SP_{3}$ problems in anisotropic and heterogeneous media and defined in \cite{bs} as

    \begin{equation}\label{4}
    \begin{array}{c}
    \partial_{t}T-\nabla\cdot(M\nabla T)=-4\pi k_{0}\tau^{-2}f(T)+k_{0}\tau^{-2}(\gamma_{2}-\gamma_{1})^{-1}(\gamma_{2}\phi_{1}-\gamma_{1}\phi_{2}), \text{\,\,\,\,on\,\,\,}
     \Omega\times[0,\text{\,}T_{f}]\\
     \text{\,}\\
    -\tau^{2}\mu_{1}^{2}\nabla\cdot(L\nabla\phi_{1})+k_{0}\phi_{1} = 4\pi k_{0}f(T), \text{\,\,\,\,on\,\,\,} \Omega\times[0,\text{\,}T_{f}]\\
    \text{\,}\\
    -\tau^{2}\mu_{2}^{2}\nabla\cdot(L\nabla\phi_{2})+k_{0}\phi_{2} = 4\pi k_{0}f(T), \text{\,\,\,\,on\,\,\,} \Omega\times[0,\text{\,}T_{f}]\\
    \end{array}
   \end{equation}
   with initial condition
   \begin{equation}\label{5}
    T(x,0)=T_{0}(x),  \text{\,\,\,\,for\,\,\,}x\in\overline{\Omega}=\Omega\cup\Gamma,
   \end{equation}
   and boundary conditions
    \begin{equation}\label{6}
    \begin{array}{c}
    \tau(M\nabla T)^{t}\vec{n}(x)+c_{m}(T-T_{m})=\alpha\pi(f(T_{m})-f(T)), \text{\,\,\,\,on\,\,\,} \Gamma\times[0,\text{\,}T_{f}]\\
     \text{\,}\\
    \tau(L\nabla\phi_{1})^{t}\vec{n}(x)+\frac{\alpha_{1}}{3}\phi_{1}=-\frac{1}{3}(\beta_{2}\phi_{2}-\eta_{1}f(T_{m})),\text{\,\,\,on\,\,} \Gamma\times[0,\text{\,}T_{f}]\\
    \text{\,}\\
    \tau(L\nabla\phi_{2})^{t}\vec{n}(x)+\frac{\alpha_{2}}{3}\phi_{2}=-\frac{1}{3}(\beta_{1}\phi_{1}-\eta_{2}f(T_{m})),\text{\,\,\,on\,\,} \Gamma\times[0,\text{\,}T_{f}]\\
    \end{array}
   \end{equation}
    where a detailed analysis of positive parameters: $\alpha_{j}$, $\beta_{j}$, $\mu_{j}$, $\eta_{j}$, and $\gamma_{j}$, for $j=1,2$, are given in \cite{bs} (see \textbf{Table 1}).\\

     \underline{\textbf{Table 1}}(Parameters and related values discussed in the $SP_{3}$ model).
    \begin{equation*}
    \begin{tabular}{cccc}
      \hline
      Parameters & Values & Parameters & Values \\
      \hline
      $\alpha_{1}$ & $2.3984$ & $\alpha_{2}$ & $1.1432$ \\
      $\beta_{1}$ & $4.71\times10^{-2}$ & $\beta_{2}$ & $1.612\times10^{-1}$ \\
      $\gamma_{1}$ & $-1.6221\times10^{4}$ & $\gamma_{2}$ & $3.0617$ \\
      $\mu_{1}$ & $5.888\times10^{-3}$ & $\mu_{2}$ & $1.4915$ \\
      $\eta_{1}$ & $3.21656\times10^{1}$ & $\eta_{2}$ & $1.49583\times10^{1}$ \\
      \hline
    \end{tabular}
    \end{equation*}

 For the sake of stability analysis and convergence of the new algorithm, the parameters $k_{0}$ and $\tau$ should be chosen so that that $k_{0}>\frac{\tau}{6}\max\{\frac{\beta_{1}^{2}\mu_{2}^{2}}{\alpha_{2}}, \text{\,}\frac{\beta_{2}^{2}\mu_{1}^{2}}{\alpha_{1}}\}$. It is worth mentioning that the initial-boundary value problem $(\ref{4})$-$(\ref{6})$ forms the mixed elliptic-parabolic PDEs and should be solved by efficient numerical methods such as the strong two-stage explicit/implicit schemes combined with the mixed FEM ($\mathcal{P}_{p}/\mathcal{P}_{p-1}$). In fact, the errors increased in the first stage when computing the predicted solution are balanced with the ones decreased in the second stage when calculating the final solution so that the stability of the numerical approach is preserved. This procedure takes advantage of providing approximate solutions with good accuracy compared to the analytical ones. Additionally, the properties of such a class of coupled equations render the $SP_{3}$ model $(\ref{4})$-$(\ref{6})$ less cost when applying an efficient computational technique than the original radiative heat transfer problem $(\ref{1})$-$(\ref{3})$. The tensors $M$ and $L$ are given by
      \begin{equation}\label{7a}
    M=P(\theta)\begin{bmatrix}
                 m_{11}(T) & 0 & 0 \\
                 0 & m_{22}(T) & 0 \\
                 0 & 0 & m_{33}(T) \\
               \end{bmatrix}P(\theta)^{-1},\text{\,\,\,\,}L=P(\theta)\begin{bmatrix}
                 \frac{1}{3(k_{0}+\sigma_{1})} & 0 & 0 \\
                 0 & \frac{1}{3(k_{0}+\sigma_{2})} & 0 \\
                 0 & 0 & \frac{1}{3(k_{0}+\sigma_{3})} \\
               \end{bmatrix}P(\theta)^{-1},
   \end{equation}
     where $P(\theta)$ is the matrix defined as
      \begin{equation}\label{7aa}
    P(\theta)=\begin{bmatrix}
                 1 & 0 & 0 \\
                 0 & \cos\theta & -\sin\theta \\
                 0 & \sin\theta & \cos\theta \\
               \end{bmatrix},
   \end{equation}
    and $m_{jj}(T)$, for $j=1,2,3$, are nonlinear diagonal elements which depend on the temperature $T$. Furthermore, the functions $m_{jj}(\cdot)$ are continuous and bounded below over $\mathbb{R}$. The nonlinearity of the system of equations $(\ref{4})$-$(\ref{6})$, is due to the fact that at high temperature, the medium conductivity depends on the temperature. When $\theta=0$, $P(0)$ returns the identity matrix, as a result the given problem $(\ref{4})$-$(\ref{6})$ generates the $SP_{3}$ model presented in \cite{9bs}. Furthermore, if a nonlinear term is present in the heat equation but not in the radiative transfer equation, the $SP_{3}$ approximations for linear isotropic situations can be applied to the nonlinear case.

       \section{Detailed description of the new computational technique}\label{sec3}
        This section considers the development of a strong two-stage explicit/implicit scheme combined with the mixed FEM ($\mathcal{P}_{p}/\mathcal{P}_{p-1}$), in an approximate solution of the $SP_{3}$ model defined by equations $(\ref{4})$-$(\ref{6})$.\\

      Let $N$ and $M_{0}$ be two positive integers. Set $\sigma=\frac{T_{f}}{N}$, be the time step and $\mathcal{T}_{\sigma}=\{t_{n}=n\sigma,\text{\,\,}0\leq n\leq N\}$ be a regular partition of the interval $[0,\text{\,}T_{f}]$. Consider $\mathcal{F}_{h}$ be a uniform partition of the domain $\overline{\Omega}=\Omega\cup\Gamma$, consisting of tetrahedra $K_{l}$ of diameter $d(K_{l})$, for $l=1,2,...,M_{0}$, where $h=\max\{d(K),\text{\,\,}K\in\mathcal{F}_{h}\}$. Moreover, $h$ is the step size of a spatial mesh of the computational domain $\overline{\Omega}$. In addition, $\mathcal{F}_{h}$ satisfies the following conditions: (i) the interior of any tetrahedron is nonempty; (ii) $int(K_{l_{1}})\cap int(K_{l_{2}})=\emptyset$, whereas $K_{l_{1}}\cap K_{l_{2}}$ is either empty or a common vertex or face, for $l_{1},l_{2}=1,2,...,M_{0}$, with $l_{1}\neq l_{2}$, where $int(K)$ means the interior of $K$; (iii) the triangulation $\mathcal{F}_{h}$ is regular while the triangulation $\mathcal{F}_{\Gamma,h}$, induced on the boundary $\Gamma=\partial\Omega$ is quasi-uniform. For the convenience of writing, we set $\psi(x,t_{n})=\psi^{n}$, as the value of the function $\psi$ at the point $x$, and at time $t_{n}$.\\

     We introduce the mixed finite element spaces $\mathcal{U}_{h}$ and $\mathcal{V}_{h}$ approximating the solutions of the $SP_{3}$ equations $(\ref{4})$-$(\ref{6})$, defined as
     \begin{equation}\label{7}
     \mathcal{U}_{h}=\{T_{h}(t)\in W_{2}^{1}(\Omega):\text{\,}T_{h}(t)|_{K}\in\mathcal{P}_{p}(K),\text{\,}\forall K\in\mathcal{F}_{h},\text{\,}\forall t\in[0,\text{\,}T_{f}]\},
     \end{equation}
     \begin{equation}\label{8}
     \mathcal{V}_{h}=\{\phi_{h}(t)=(\phi_{1h}(t),\phi_{2h}(t))^{t}\in[W_{2}^{1}(\Omega)]^{2}:\text{\,}\phi_{jh}(t)|_{K}\in\mathcal{P}_{p-1}(K),\text{\,\,}\forall K\in\mathcal{F}_{h}, \text{\,\,}\forall t\in[0,\text{\,}T_{f}],\text{\,\,}j=1,2\},
      \end{equation}
      where $w^{t}$ means the transpose of the vector $w$, $\mathcal{P}_{l}(K)$ represents the space of polynomials defined on $K$ with degree less than or equal $l$ and $W_{2}^{1}(\Omega)$ is the Sobolev space defined as
      \begin{equation*}
      W_{2}^{1}(\Omega)=\{u\in L^{2}(\Omega),\text{\,}\nabla u\in[L^{2}(\Omega)]^{3}\}.
     \end{equation*}

     The spaces $L^{2}(\Omega)$ and $[L^{2}(\Omega)]^{3}$ are equipped with the inner products $\left(\cdot,\cdot\right)_{0}$ and $\left(\cdot,\cdot\right)_{\bar{0}}$, respectively, defined by
       \begin{equation}\label{10}
        \left(u_{1},u_{2}\right)_{0}=\int_{\Omega}u_{1}u_{2}d\Omega,\text{\,\,}\forall u_{1},u_{2}\in L^{2}(\Omega),\text{\,\,\,}
     \left(u,v\right)_{\bar{0}}=\int_{\Omega}u^{t}v d\Omega,\text{\,\,for\,\,} u,v\in [L^{2}(\Omega)]^{3}.
     \end{equation}

      Consider the bilinear operators $a(\cdot,\cdot)$, $<\cdot,\cdot>$, $\left(\cdot,\cdot\right)_{\Delta}$, $A(\cdot,\cdot)$ and $<\cdot,\cdot>_{\Delta}$, defined as
        \begin{equation*}
        a(u_{1},u_{2})=\left(M\nabla u_{1},\nabla u_{2}\right)_{\bar{0}},\text{\,\,}<u_{1},u_{2}>=\int_{\Gamma}u_{1}u_{2}d\Gamma,\text{\,\,}\forall u_{1},u_{2}\in W_{2}^{1}(\Omega),\text{\,\,}
       \left(u,v\right)_{\Delta}=\begin{pmatrix}
                                 \left(u_{1},v_{1}\right)_{0} \\
                                 \left(u_{2},v_{2}\right)_{0} \\
                               \end{pmatrix},
     \end{equation*}
       \begin{equation}\label{11}
       A(u,v)=\begin{pmatrix}
                                 \left(L\nabla u_{1},\nabla v_{1}\right)_{\bar{0}} \\
                                 \left(L\nabla u_{2},\nabla v_{2}\right)_{\bar{0}} \\
                               \end{pmatrix},\text{\,\,}<u,v>_{\Delta}=\begin{pmatrix}
                                 <u_{1},v_{1}> \\
                                 <u_{2},v_{2}> \\
                               \end{pmatrix},\text{\,\,for\,\,} u=(u_{1},u_{2})^{t},v=(v_{1},v_{2})^{t}\in [W_{2}^{1}(\Omega)]^{2}.
      \end{equation}

      The following integration by parts will play an important role in our analysis
      \begin{equation}\label{12}
      \left(\Delta u_{1},u_{2}\right)_{0}=\int_{\Gamma}u_{2}(\nabla u_{1})^{t}\vec{n}d\Gamma-\left(\nabla u_{1},\nabla u_{2}\right)_{\bar{0}},\text{\,\,}
       \left(\bar{\Delta} v,w\right)_{\Delta}=\begin{pmatrix}
                                 \int_{\Gamma}w_{1}(\nabla v_{1})^{t}\vec{n}d\Gamma \\
                                 \text{\,}
                                 \int_{\Gamma}w_{2}(\nabla v_{2})^{t}\vec{n}d\Gamma \\
                               \end{pmatrix}-\begin{pmatrix}
                                 \left(\nabla v_{1},\nabla w_{1}\right)_{\bar{0}} \\
                                 \text{\,}
                                 \left(\nabla v_{2},\nabla w_{2}\right)_{\bar{0}} \\
                               \end{pmatrix},
      \end{equation}
       where $\bar{\Delta} v=(\Delta v_{1},\Delta v_{2})^{t}$, for $u_{1}\in H^{2}(\Omega)$, $u_{2}\in W_{2}^{1}(\Omega)$, $v=(v_{1},v_{2})^{t}\in[H^{2}(\Omega)]^{2}$ and
       $w=(w_{1},w_{2})^{t}\in[W_{2}^{1}(\Omega)]^{2}$, where $\vec{n}$ designates the unit outward normal vector on $\Gamma$.\\

      Now, we should provide the weak formulation of the boundary-value problem $(\ref{4})$ and $(\ref{6})$. Let $(w,(\psi_{1},\psi_{2})^{t})\in[W_{2}^{1}(\Omega)]^{3}$. Multiplying the first, second and third equations in system $(\ref{4})$ by $w$, $\psi_{1}$ and $\psi_{2}$, respectively, and using the scalar products defined in equation $(\ref{10})$ together with the integration by parts $(\ref{12})$ to get
     \begin{equation*}
    \begin{array}{c}
    \left(\partial_{t}T,w\right)_{0}+a(T,w)=\int_{\Gamma}w(M\nabla T)^{t}\vec{n}d\Gamma-4\pi k_{0}\tau^{-2}\left(f(T),w\right)_{0}+k_{0}\tau^{-2}(\gamma_{2}-\gamma_{1})^{-1}\left((\gamma_{2}\phi_{1}-\gamma_{1}\phi_{2}),w\right)_{0},\\
     \text{\,}\\
    \tau^{2}\mu_{1}^{2}\left(L\nabla\phi_{1},\nabla \psi_{1}\right)_{\bar{0}}+k_{0}\left(\phi_{1},\psi_{1}\right)_{0}=\tau^{2}\mu_{1}^{2}\int_{\Gamma}\psi_{1}(L\nabla \phi_{1})^{t}\vec{n}d\Gamma + 4\pi k_{0}\left(f(T),\psi_{1}\right)_{0},\\
    \text{\,}\\
    \tau^{2}\mu_{2}^{2}\left(L\nabla\phi_{2},\nabla \psi_{2}\right)_{0}+k_{0}\left(\phi_{2},\psi_{2}\right)_{0}=\tau^{2}\mu_{2}^{2}\int_{\Gamma}\psi_{2}(L\nabla \phi_{2})^{t}\vec{n}d\Gamma+4\pi k_{0}\left(f(T),\psi_{2}\right)_{0}.\\
    \end{array}
   \end{equation*}

    Utilizing the boundary conditions $(\ref{6})$ and the scalar product $<\cdot,\cdot>$, given in relation $(\ref{11})$, these equations are equivalent to
    \begin{equation}\label{13}
    \begin{array}{c}
    \left(\partial_{t}T,w\right)_{0}+a(T,w)=\tau^{-1}<c_{m}(T_{m}-T)+\alpha\pi(f(T_{m})-f(T)),w>-4\pi k_{0}\tau^{-2}\left(f(T),w\right)_{0}+\\ k_{0}\tau^{-2}(\gamma_{2}-\gamma_{1})^{-1}\left((\gamma_{2}\phi_{1}-\gamma_{1}\phi_{2}),w\right)_{0},\\
     \text{\,}\\
    \tau^{2}\mu_{1}^{2}\left(L\nabla\phi_{1},\nabla \psi_{1}\right)_{\bar{0}}+k_{0}\left(\phi_{1},\psi_{1}\right)_{0}= -\frac{\tau\mu_{1}^{2}}{3}<\alpha_{1}\phi_{1}+\beta_{2}\phi_{2}-\eta_{1}f(T_{m}),\psi_{1}> +4\pi k_{0}\left(f(T),\psi_{1}\right)_{0},\\
    \text{\,}\\
    \tau^{2}\mu_{2}^{2}\left(L\nabla\phi_{2},\nabla \psi_{2}\right)_{0}+k_{0}\left(\phi_{2},\psi_{2}\right)_{0}=
     -\frac{\tau\mu_{2}^{2}}{3}<\alpha_{2}\phi_{2}+\beta_{1}\phi_{1}-\eta_{2}f(T_{m}),\psi_{2}>+4\pi k_{0}\left(f(T),\psi_{2}\right)_{0}.\\
    \end{array}
   \end{equation}

   Setting
     \begin{equation*}
      \psi=\begin{pmatrix}
             \psi_{1} \\
             \psi_{2} \\
           \end{pmatrix},\text{\,\,}\phi=\begin{pmatrix}
                                         \phi_{1} \\
                                         \phi_{2} \\
                                       \end{pmatrix},\text{\,\,}\overline{\mu}=\begin{pmatrix}
                                                                               \mu_{1} \\
                                                                               \mu_{2} \\
                                                                             \end{pmatrix},\text{\,\,}\overline{\eta}=\begin{pmatrix}
                                                                                                                      \eta_{1} \\
                                                                                                                      \eta_{2} \\
                                                                                                                    \end{pmatrix},\text{\,\,}\overline{\alpha}=\begin{pmatrix}
                                                                                                                      \alpha_{1} \\
                                                                                                                      \alpha_{2} \\
                                                                                                                    \end{pmatrix},\text{\,\,}\overline{\beta}=\begin{pmatrix}
                                                                                                                      \beta_{1} \\
                                                                                                                      \beta_{2} \\
                                                                                                                    \end{pmatrix},
     \end{equation*}
       \begin{equation}\label{not}
      F(u)=\begin{pmatrix}
             f(u) \\
             f(u) \\
           \end{pmatrix},\text{\,\,}G_{\overline{\alpha}\overline{\beta}}(\phi)=\begin{pmatrix}
                                                      \alpha_{1} & \beta_{2} \\
                                                      \beta_{1} & \alpha_{2} \\
                                                    \end{pmatrix}\phi,
      \end{equation}
     and using the operators $A(\cdot,\cdot)$, $\left(\cdot,\cdot\right)_{\Delta}$, and $<\cdot,\cdot>_{\Delta}$, defined in relation $(\ref{11})$, system $(\ref{13})$ becomes
     \begin{equation}\label{14}
    \begin{array}{c}
    \left(\partial_{t}T,w\right)_{0}+a(T,w)=\tau^{-1}<c_{m}(T_{m}-T)+\alpha\pi(f(T_{m})-f(T)),w>-4\pi k_{0}\tau^{-2}\left(f(T),w\right)_{0}+\\ k_{0}\tau^{-2}(\gamma_{2}-\gamma_{1})^{-1}\left(\gamma_{2}\phi_{1}-\gamma_{1}\phi_{2},w\right)_{0},\text{\,\,\,\,}\forall w\in W_{2}^{1}(\Omega),\\
     \text{\,}\\
    \tau^{2}\overline{\mu}^{.2}.*A(\phi,\psi)+k_{0}\left(\phi,\psi\right)_{\Delta}=-\frac{\tau}{3}\overline{\mu}^{.2}.*(<G_{\overline{\alpha}\overline{\beta}}(\phi),\psi>_{\Delta}
    -\overline{\eta}.*<F(T_{m}),\psi>_{\Delta})+\\
     4\pi k_{0}\left(F(T),\psi\right)_{\Delta},\text{\,\,\,\,}\forall \psi=(\psi_{1},\psi_{2})^{t}\in [W_{2}^{1}(\Omega)]^{2},\\
    \end{array}
   \end{equation}
    where $"^{.2}"$ denotes the element-by-element square power and $".*"$ means the element-by-element multiplication.\\

   To construct the new method, we should approximate the function $T(t)$ at the discrete times: $t_{n-\frac{1}{2}}$, $t_{n}$ and $t_{n+\frac{1}{2}}$ (respectively, $t_{n}$, $t_{n+\frac{1}{2}}$ and $t_{n+1}$). The interpolation of $T(t)$ at points: $t_{n-\frac{1}{2}}$, $t_{n}$ and $t_{n+\frac{1}{2}}$, gives
      \begin{equation*}
        T(t)=\frac{(t-t_{n-\frac{1}{2}})(t-t_{n})}{(t_{n+\frac{1}{2}}-t_{n-\frac{1}{2}})(t_{n+\frac{1}{2}}-t_{n})}T^{n+\frac{1}{2}}+
        \frac{(t-t_{n-\frac{1}{2}})(t-t_{n+\frac{1}{2}})}{(t_{n}-t_{n-\frac{1}{2}})(t_{n}-t_{n+\frac{1}{2}})}T^{n}+
       \end{equation*}
       \begin{equation*}
       \frac{(t-t_{n})(t-t_{n+\frac{1}{2}})}{(t_{n-\frac{1}{2}}-t_{n})(t_{n-\frac{1}{2}}-t_{n+\frac{1}{2}})}T^{n-\frac{1}{2}}+\frac{1}{6}(t-t_{n-\frac{1}{2}})(t-t_{n})
         (t-t_{n+\frac{1}{2}})\partial_{ttt}T(\epsilon(t)),
       \end{equation*}
       where $\epsilon(t)$ is between the minimum and maximum of $t_{n-\frac{1}{2}}$, $t_{n}$, $t_{n+\frac{1}{2}}$ and $t$. The time derivative provides
       \begin{equation*}
        \partial_{t}T(t)=\frac{2}{\sigma^{2}}[(2t-t_{n-\frac{1}{2}}-t_{n})T^{n+\frac{1}{2}}-2(2t-t_{n-\frac{1}{2}}-t_{n+\frac{1}{2}})T^{n}+(2t-t_{n}-t_{n+\frac{1}{2}})T^{n-\frac{1}{2}}]+
       \end{equation*}
       \begin{equation}\label{15}
       \frac{1}{6}\{(t-t_{n-\frac{1}{2}})(t-t_{n})(t-t_{n+\frac{1}{2}})\frac{d}{dt}(\partial_{ttt}T(\epsilon(t)))+\partial_{ttt}T(\epsilon(t))\frac{d}{dt}
       [(t-t_{n-\frac{1}{2}})(t-t_{n})(t-t_{n+\frac{1}{2}})]\}.
       \end{equation}

       For $t=t_{n}$, equation $(\ref{15})$ provides
       \begin{equation}\label{16}
        \partial_{t}T^{n}=\frac{1}{\sigma}(T^{n+\frac{1}{2}}-T^{n-\frac{1}{2}})+\frac{\sigma^{2}}{12}\partial_{ttt}T(\epsilon^{n}).
       \end{equation}

       Replacing into the system of equations $(\ref{14})$ $t$ with $t_{n}$, substituting equation $(\ref{16})$ into the new equation  and rearranging terms, we obtain
       \begin{equation}\label{17}
    \begin{array}{c}
    \left(T^{n+\frac{1}{2}},w\right)_{0}+\sigma a(T^{n},w)=\left(T^{n-\frac{1}{2}},w\right)_{0}
     +\sigma\tau^{-1}<c_{m}(T_{m}-T^{n})+\alpha\pi(f(T_{m})-f(T^{n})),w>-\\ 4\pi k_{0}\tau^{-2}\sigma\left(f(T^{n}),w\right)_{0}+ k_{0}\tau^{-2}(\gamma_{2}-\gamma_{1})^{-1}\sigma\left(\gamma_{2}\phi_{1}^{n}-\gamma_{1}\phi_{2}^{n},w\right)_{0}-
     \frac{\sigma^{3}}{12}\left(\partial_{ttt}T(\epsilon^{n}),w\right)_{0},\text{\,\,\,\,}\forall w\in W_{2}^{1}(\Omega),\\
     \text{\,}\\
    \tau^{2}\overline{\mu}^{.2}.*A(\phi^{n},\psi)+k_{0}\left(\phi^{n},\psi\right)_{\Delta}=-\frac{\tau}{3}\overline{\mu}^{.2}.*(<G_{\overline{\alpha}\overline{\beta}}(\phi^{n}),\psi>_{\Delta}
    -\overline{\eta}.*<F(T_{m}),\psi>_{\Delta})+\\
     4\pi k_{0}\left(F(T^{n}),\psi\right)_{\Delta},\text{\,\,\,\,}\forall \psi=(\psi_{1},\psi_{2})^{t}\in [W_{2}^{1}(\Omega)]^{2}.\\
    \end{array}
   \end{equation}

    Tracking the error term $-\frac{\sigma^{3}}{12}\left(\partial_{ttt}T(\epsilon^{n}),w\right)_{0}$, into the first equation in system $(\ref{17})$ and replacing the analytical solution $(T,\phi)=(T,(\phi_{1},\phi_{2}))\in H^{4}(0,T_{f};\text{\,}W_{2}^{2}(\Omega))\times L^{2}(0,T_{f};\text{\,}[W_{2}^{2}(\Omega)]^{2})$, with the approximate one $(T_{h}(t),\phi_{h}(t))=(T_{h}(t),(\phi_{1h}(t),\phi_{2h}(t)))\in\mathcal{U}_{h}\times\mathcal{V}_{h}$, for every $t\in[0,\text{\,}T_{f}]$, where $\mathcal{U}_{h}$ and $\mathcal{V}_{h}$ are given by equations $(\ref{7})$ and $(\ref{8})$, respectively, to get the first stage of the proposed computational technique, that is, for $n=0,1,...,N-1$,
       \begin{equation}\label{18}
    \begin{array}{c}
    \left(T_{h}^{n+\frac{1}{2}},w\right)_{0}+\sigma a(T_{h}^{n},w)=\left(T_{h}^{n-\frac{1}{2}},w\right)_{0}
     +\sigma\tau^{-1}<c_{m}(T_{m}-T_{h}^{n})+\alpha\pi(f(T_{m})-f(T_{h}^{n})),w>\\ -4\pi k_{0}\tau^{-2}\sigma\left(f(T_{h}^{n}),w\right)_{0}+ k_{0}\tau^{-2}(\gamma_{2}-\gamma_{1})^{-1}\sigma\left(\gamma_{2}\phi_{1h}^{n}-\gamma_{1}\phi_{2h}^{n},w\right)_{0},\text{\,\,\,\,}\forall w\in W_{2}^{1}(\Omega),\\
     \text{\,}\\
    \tau^{2}\overline{\mu}^{.2}.*A(\phi_{h}^{n},\psi)+k_{0}\left(\phi_{h}^{n},\psi\right)_{\Delta}=
    -\frac{\tau}{3}\overline{\mu}^{.2}.*(<G_{\overline{\alpha}\overline{\beta}}(\phi_{h}^{n}),\psi>_{\Delta}-\overline{\eta}.*<F(T_{m}),\psi>_{\Delta})\\
     +4\pi k_{0}\left(F(T_{h}^{n}),\psi\right)_{\Delta},\text{\,\,\,\,}\forall \psi=(\psi_{1},\psi_{2})^{t}\in [W_{2}^{1}(\Omega)]^{2}.
    \end{array}
   \end{equation}

    In a similar way, approximating the function $T(t)$ at the discrete times: $t_{n}$, $t_{n+\frac{1}{2}}$ and $t_{n+1}$, it is not hard to show that
       \begin{equation*}
        \partial_{t}T(t)=\frac{2}{\sigma^{2}}[(2t-t_{n}-t_{n+\frac{1}{2}})T^{n+1}-2(2t-t_{n}-t_{n+1})T^{n+\frac{1}{2}}+(2t-t_{n+\frac{1}{2}}-t_{n+1})T^{n}]+
       \end{equation*}
       \begin{equation*}
       \frac{1}{6}\{(t-t_{n})(t-t_{n+\frac{1}{2}})(t-t_{n+1})\frac{d}{dt}(\partial_{ttt}T(\overline{\epsilon}(t)))+\partial_{ttt}T(\overline{\epsilon}(t))\frac{d}{dt}
       [(t-t_{n})(t-t_{n+\frac{1}{2}})(t-t_{n+1})]\},
       \end{equation*}
        where $\overline{\epsilon}(t)$ is between the minimum and maximum of $t_{n}$, $t_{n+\frac{1}{2}}$, $t_{n+1}$ and $t$.
       For $t=t_{n+1}$, this equation results in
       \begin{equation}\label{19}
        \partial_{t}T^{n+1}=\frac{1}{\sigma}(3T^{n+1}-4T^{n+\frac{1}{2}}+T^{n})+\frac{\sigma^{2}}{12}\partial_{ttt}T(\overline{\epsilon}^{n+1}).
       \end{equation}

       Replacing into system $(\ref{14})$, $t$ with $t_{n+1}$, combining the obtained system with equation $(\ref{19})$ and rearranging terms, this gives
       \begin{equation}\label{20}
    \begin{array}{c}
    3\left(T^{n+1},w\right)_{0}+\sigma a(T^{n+1},w)=\left(4T^{n+\frac{1}{2}}-T^{n},w\right)_{0}
     +\sigma\tau^{-1}<c_{m}(T_{m}-T^{n+1})+\\ \alpha\pi(f(T_{m})-f(T^{n+1})),w>- 4\pi k_{0}\tau^{-2}\sigma\left(f(T^{n+1}),w\right)_{0}+k_{0}\tau^{-2}(\gamma_{2}-\gamma_{1})^{-1}\sigma*\\
     \left(\gamma_{2}\phi_{1}^{n+1}-\gamma_{1}\phi_{2}^{n+1},w\right)_{0}-\frac{\sigma^{3}}{12}\left(\partial_{ttt}T(\overline{\epsilon}^{n+1}),w\right)_{0},
      \text{\,\,\,\,}\forall w\in W_{2}^{1}(\Omega),\\
     \text{\,}\\
    \tau^{2}\overline{\mu}^{.2}.*A(\phi^{n+1},\psi)+k_{0}\left(\phi^{n+1},\psi\right)_{\Delta}=-\frac{\tau}{3}\overline{\mu}^{.2}.*(<G_{\overline{\alpha}\overline{\beta}}(\phi^{n+1}),\psi>_{\Delta}
    -\overline{\eta}.*<F(T_{m}),\psi>_{\Delta})+\\
     4\pi k_{0}\left(F(T^{n+1}),\psi\right)_{\Delta},\text{\,\,\,\,}\forall \psi=(\psi_{1},\psi_{2})^{t}\in [W_{2}^{1}(\Omega)]^{2}.\\
    \end{array}
   \end{equation}

    Truncating the error term $-\frac{\sigma^{3}}{12}\left(\partial_{ttt}T(\overline{\epsilon}^{n+1}),w\right)_{0}$, into the first equation in system $(\ref{20})$ and replacing the exact solution $(T,\phi)\in H^{4}(0,T_{f};\text{\,}W_{2}^{2}(\Omega))\times L^{2}(0,T_{f};\text{\,}[W_{2}^{2}(\Omega)]^{2})$, with the computed one $(T_{h}(t),\phi_{h}(t))\in\mathcal{U}_{h}\times\mathcal{V}_{h}$, for every $t\in[0,\text{\,}T_{f}]$, to obtain the second step of the developed numerical approach, that is, for $n=0,1,...,N-1$,
      \begin{equation}\label{21}
    \begin{array}{c}
    3\left(T_{h}^{n+1},w\right)_{0}+\sigma a(T_{h}^{n+1},w)=\left(4T_{h}^{n+\frac{1}{2}}-T_{h}^{n},w\right)_{0}
     +\sigma\tau^{-1}<c_{m}(T_{m}-T_{h}^{n+1})+\\ \alpha\pi(f(T_{m})-f(T_{h}^{n+1})),w>- 4\pi k_{0}\tau^{-2}\sigma\left(f(T_{h}^{n+1}),w\right)_{0}+k_{0}\tau^{-2}(\gamma_{2}-\gamma_{1})^{-1}\sigma*\\
     \left(\gamma_{2}\phi_{1_{h}}^{n+1}-\gamma_{1}\phi_{2_{h}}^{n+1},w\right)_{0},\text{\,\,\,\,}\forall w\in W_{2}^{1}(\Omega),\\
     \text{\,}\\
    \tau^{2}\overline{\mu}^{.2}.*A(\phi_{h}^{n+1},\psi)+k_{0}\left(\phi_{h}^{n+1},\psi\right)_{\Delta}=-\frac{\tau}{3}\overline{\mu}^{.2}.*(<G_{\overline{\alpha}\overline{\beta}}(\phi_{h}^{n+1}),\psi>_{\Delta}
    -\overline{\eta}.*<F(T_{m}),\psi>_{\Delta})+\\
     4\pi k_{0}\left(F(T_{h}^{n+1}),\psi\right)_{\Delta},\text{\,\,\,\,}\forall \psi=(\psi_{1},\psi_{2})^{t}\in [W_{2}^{1}(\Omega)]^{2}.\\
    \end{array}
   \end{equation}

   Let $Q_{h}$ be the $L^{2}$-projection operator from $L^{2}(\Omega)$ onto $\mathcal{U}_{h}$. So, $Q_{h}$ satisfies the following property
         \begin{equation*}
     \left(Q_{h}u_{1},u_{2}\right)_{0}=\left(u_{1},u_{2}\right)_{0},\text{\,\,\,\,}\forall u_{1},u_{2}\in L^{2}(\Omega).
         \end{equation*}

    To start the new algorithm, the first two terms $T_{h}^{0}$ and $T_{h}^{-\frac{1}{2}}$ are needed. However, the first term $T_{h}^{0}$ satisfies
     \begin{equation}\label{22}
     T_{h}^{0}=Q_{h}T_{0},
     \end{equation}
      where $T_{0}$ is the initial condition given by equation $(\ref{5})$, while the second term $T_{h}^{-\frac{1}{2}}$, should be approximated using the first equation in system $(\ref{4})$. In fact, applying the Taylor series expansion to obtain
      \begin{equation}\label{23}
     T^{-\frac{1}{2}}\approx T_{0}-\frac{\sigma}{2}\partial_{t}T(t_{0})=T_{0}-\frac{\sigma}{2}[\nabla\cdot(M^{0}\nabla T_{0})-4\pi k_{0}\tau^{-2}f(T_{0})+k_{0}\tau^{-2}(\gamma_{2}-\gamma_{1})^{-1}(\gamma_{2}\phi_{1}^{0}-\gamma_{1}\phi_{2}^{0})],
     \end{equation}
    where the error term in approximation $(\ref{23})$, is $\frac{\sigma^{2}}{8}\partial_{tt}T(\epsilon_{0})$, with $0<\epsilon_{0}<\frac{\sigma}{2}$. The terms $\phi_{1}^{0}$ and $\phi_{2}^{0}$ satisfy
     \begin{equation}\label{24}
    \begin{array}{c}
    -\tau^{2}\mu_{1}^{2}\nabla\cdot(L\nabla\phi_{1}^{0})+k_{0}\phi_{1}^{0}=4\pi k_{0}f(T_{0}) \text{\,\,\,and\,\,\,}
    -\tau^{2}\mu_{2}^{2}\nabla\cdot(L\nabla\phi_{2}^{0})+k_{0}\phi_{2}^{0}=4\pi k_{0}f(T_{0}),\text{\,\,\,\,on\,\,\,\,}\Omega.\\
    \end{array}
   \end{equation}

   Thus, take
   \begin{equation}\label{25}
     T_{h}^{-\frac{1}{2}}=Q_{h}T^{-\frac{1}{2}}.
     \end{equation}

    The weak formulation corresponding to equation $(\ref{24})$ is given by
     \begin{equation*}
    \tau^{2}\overline{\mu}^{.2}.*A(\phi^{0},\psi)+k_{0}\left(\phi^{0},\psi\right)_{\Delta}+\frac{\tau}{3}\overline{\mu}^{.2}.*<G_{\overline{\alpha}\overline{\beta}}(\phi^{0}),\psi>_{\Delta}
    =\frac{\tau}{3}\overline{\mu}^{.2}.*\overline{\eta}.*<F(T_{m}),\psi>_{\Delta}+
     \end{equation*}
     \begin{equation}\label{26}
     4\pi k_{0}\left(F(T_{0}),\psi\right)_{\Delta},\text{\,\,\,\,}\forall \psi=(\psi_{1},\psi_{2})^{t}\in [W_{2}^{1}(\Omega)]^{2}.
   \end{equation}

   Plugging the systems of equations $(\ref{18})$, $(\ref{21})$ and equations $(\ref{22})$, $(\ref{25})$ to get the desired two-stage explicit/implicit method combined with the mixed finite element technique ($\mathcal{P}_{p}/\mathcal{P}_{p-1}$), for solving the initial-boundary value problem $(\ref{4})$-$(\ref{6})$, that is, given $(T_{h}^{n},\phi_{h}^{n})\in\mathcal{U}_{h}\times\mathcal{V}_{h}$, find $(T_{h}^{n+1},\phi_{h}^{n+1})\in\mathcal{U}_{h}\times\mathcal{V}_{h}$, for $n=0,1,...N-1$, so that
     \begin{equation}\label{s1}
    \begin{array}{c}
    \left(T_{h}^{n+\frac{1}{2}},w\right)_{0}+\sigma a(T_{h}^{n},w)=\left(T_{h}^{n-\frac{1}{2}},w\right)_{0}
     +\sigma\tau^{-1}<c_{m}(T_{m}-T_{h}^{n})+\alpha\pi(f(T_{m})-f(T_{h}^{n})),w>\\ -4\pi k_{0}\tau^{-2}\sigma\left(f(T_{h}^{n}),w\right)_{0}+ k_{0}\tau^{-2}(\gamma_{2}-\gamma_{1})^{-1}\sigma\left(\gamma_{2}\phi_{1h}^{n}-\gamma_{1}\phi_{2h}^{n},w\right)_{0},\text{\,\,\,\,}\forall w\in W_{2}^{1}(\Omega),\\
     \text{\,}\\
    \tau^{2}\overline{\mu}^{.2}.*A(\phi_{h}^{n},\psi)+k_{0}\left(\phi_{h}^{n},\psi\right)_{\Delta}+\frac{\tau}{3}\overline{\mu}^{.2}.*<G_{\overline{\alpha}\overline{\beta}}(\phi_{h}^{n}),\psi>_{\Delta}=
\frac{\tau}{3}\overline{\mu}^{.2}.*\overline{\eta}.*<F(T_{m}),\psi>_{\Delta}+\\
     4\pi k_{0}\left(F(T_{h}^{n}),\psi\right)_{\Delta},\text{\,\,\,\,}\forall \psi=(\psi_{1},\psi_{2})^{t}\in [W_{2}^{1}(\Omega)]^{2}.
    \end{array}
   \end{equation}
     \begin{equation}\label{s2}
    \begin{array}{c}
    3\left(T_{h}^{n+1},w\right)_{0}+\sigma a(T_{h}^{n+1},w)=\left(4T_{h}^{n+\frac{1}{2}}-T_{h}^{n},w\right)_{0}
     +\sigma\tau^{-1}<c_{m}(T_{m}-T_{h}^{n+1})+\\ \alpha\pi(f(T_{m})-f(T_{h}^{n+1})),w>- 4\pi k_{0}\tau^{-2}\sigma\left(f(T_{h}^{n+1}),w\right)_{0}+k_{0}\tau^{-2}(\gamma_{2}-\gamma_{1})^{-1}\sigma*\\
     \left(\gamma_{2}\phi_{1_{h}}^{n+1}-\gamma_{1}\phi_{2_{h}}^{n+1},w\right)_{0},\text{\,\,\,\,}\forall w\in W_{2}^{1}(\Omega),\\
     \text{\,}\\
    \tau^{2}\overline{\mu}^{.2}.*A(\phi_{h}^{n+1},\psi)+k_{0}\left(\phi_{h}^{n+1},\psi\right)_{\Delta}+\frac{\tau}{3}\overline{\mu}^{.2}.*<G_{\overline{\alpha}\overline{\beta}}
    (\phi_{h}^{n+1}),\psi>_{\Delta}=\frac{\tau}{3}\overline{\mu}^{.2}.*\overline{\eta}.*<F(T_{m}),\psi>_{\Delta}+\\
     4\pi k_{0}\left(F(T_{h}^{n+1}),\psi\right)_{\Delta},\text{\,\,\,\,}\forall \psi=(\psi_{1},\psi_{2})^{t}\in [W_{2}^{1}(\Omega)]^{2}.\\
    \end{array}
   \end{equation}
    with initial condition
    \begin{equation}\label{s3}
     T_{h}^{0}=Q_{h}T_{0},\text{\,\,\,\,\,\,}T_{h}^{-\frac{1}{2}}=Q_{h}T^{-\frac{1}{2}},
     \end{equation}
    where the term $T^{-\frac{1}{2}}$, is given by approximation $(\ref{23})$.

       It's worth recalling that the first stage of the new algorithm given by equation $(\ref{s1})$ deals with an explicit numerical approach whereas the second stage defined by equation $(\ref{s2})$ works with an implicit numerical method. Moreover, the developed computational technique $(\ref{s1})$-$(\ref{s3})$ should be observed as a predictor-corrector method. As a result, the errors increased at the predictor step are balanced by the ones decreased at the corrector step, so that the stability of the proposed algorithm $(\ref{s1})$-$(\ref{s3})$ is preserved.

     \section{Stability analysis of the developed computational approach}\label{sec4}

     \text{\,\,\,\,\,\,\,\,\,\,}This section deals with the analysis of stability of the proposed numerical scheme $(\ref{s1})$-$(\ref{s3})$, for solving the $SP_{3}$ model $(\ref{4})$-$(\ref{6})$ corresponding to the nonlinear radiative heat transfer problem $(\ref{1})$-$(\ref{3})$. The second stage given by equation $(\ref{s2})$ is an implicit scheme, hence it should be unconditionally stable in the $L^{2}(\Omega)$-norm thanks to the properties of numerical implicit methods. As a result, the stability analysis of the constructed strategy $(\ref{s1})$-$(\ref{s3})$, will be restricted to the first stage of the new algorithm. Because the formulas can become quite heavy and for the convenience of writing, only a necessary condition on the time step restriction which ensures the stability of the first stage of the numerical method will be considered. However, the analysis of the stability under a suitable time step limitation is too difficult because of the nonlinear behavior of the $SP_{3}$ problem. Finally, a combination of first equations in systems $(\ref{17})$, $(\ref{20})$ and approximation $(\ref{23})$ indicates that the new algorithm $(\ref{s1})$-$(\ref{s3})$ is temporal second-order accurate whenever it is stable. This result will be confirmed in the numerical experiments section.\\

    We introduce the following norms: $\|\cdot\|_{0}$, $\|\cdot\|_{\bar{0}}$, $\|\cdot\|_{>}$, associated with the scalar products $\left(\cdot,\cdot\right)_{0}$, $\left(\cdot,\cdot\right)_{\bar{0}}$, and $<\cdot,\cdot>$, respectively. These norms are defined as
     \begin{equation}\label{27}
      \|u\|_{0}=\sqrt{\left(u,u\right)_{0}},\text{\,\,\,}\|u\|_{>}=\sqrt{<u,u>},\text{\,\,\,}\forall u\in L^{2}(\Omega), \text{\,\,\,}\|w\|_{\bar{0}}=\sqrt{\left(w,w\right)_{\bar{0}}},\text{\,\,\,}\forall w=(w_{1},w_{2},w_{3})^{t}\in[L^{2}(\Omega)]^{3}.
      \end{equation}

    We assume that the sequences $\{\mathcal{U}_{h}\}_{h>0}$ and $\{\mathcal{V}_{h}\}_{h>0}$, of mixed finite element subspaces are used in the fluid regions. The analysis of stability of the developed algorithm requires the following Poincar\'{e}-Friedrich and trace inequalities $(\ref{30})$ and $(\ref{31})$, respectively.
      \begin{equation}\label{30}
      \|w\|_{0}\leq \widehat{C}_{1}\|\nabla w\|_{\bar{0}},\text{\,\,\,\,}w\in W_{2}^{1}(\Omega),
      \end{equation}
       \begin{equation}\label{31}
      \|w\|_{>}\leq \sqrt{\widehat{C}_{1}}\|w\|_{0}^{\frac{1}{2}}\|\nabla w\|_{\bar{0}}^{\frac{1}{2}},\text{\,\,\,\,}w\in W_{2}^{1}(\Omega),
      \end{equation}
      where $\widehat{C}_{1}$ is a positive constant that does not depend on the time step $\sigma$ and grid space $h$.\\

      It's worth mentioning that the considered $SP_{3}$ problem $(\ref{4})$-$(\ref{6})$ is too complex, as a result, find an appropriate condition on time step for stability of the proposed numerical approach $(\ref{s1})$-$(\ref{s3})$ should be too difficult. Because $f(T_{h}^{n})=c_{bs}(T_{h}^{n})^{4}$, where $c_{bs}=5.67\times10^{-8}$ and the term $T_{h}^{n}$ is computed at the second stage which is an unconditionally stable scheme, the sequence of approximate spectral intensity of the black-body $\{\|f(T_{h}^{n})\|_{>}\}_{n\leq N}$, cannot not grow faster than the sequence $\{\sqrt{N}\}_{N\geq1}$. This leads to the following necessary restriction for stability of the first stage of the proposed approach
     \begin{equation}\label{ts}
      \underset{\sigma\rightarrow0}{\lim}\sqrt{\sigma}\|f(T_{h}^{n})\|_{>}=0,
      \end{equation}
      for $n=0,1,...,N$, where $\sigma=\frac{T_{f}}{N}$. Condition $(\ref{ts})$ implies that the proposed method cannot advance the approximate solution with a maximum
      allowable time step. This restriction is highly intriguing since the first and second stages of the new
      computational technique should employ the same time step. This is because the second step, which works with an implicit scheme, requires a sufficiently small time step to
      guarantee both stability and convergence.\\

     \begin{lemma}\label{l1}
     For any $w,\rho\in W_{2}^{1}(\Omega)$, it holds
     \begin{equation}\label{32}
     |<w,\rho>|\leq \epsilon\|w\|_{>}^{2}+\frac{1}{4\epsilon}\|\rho\|_{>}^{2}\text{\,\,\,\,\,and\,\,\,\,\,}|<w,\rho>|\leq \epsilon_{1}\|w\|_{>}^{2}+\frac{1}{4\epsilon_{1}}\widehat{C}_{1}^{2}\|\nabla\rho\|_{\bar{0}}^{2},
     \end{equation}
     for every positive constants $\epsilon$ and $\epsilon_{1}$.
     \end{lemma}

   \begin{proof}
   It follows from the definition of the scalar product $<\cdot,\cdot>$ given in relation $(\ref{11})$ together with the Poincar\'{e}-Friedrich and trace inequalities $(\ref{30})$ and $(\ref{31})$, respectively, and the Cauchy-Schwarz inequality that
    \begin{equation*}
    |<w,\rho>|\leq |\int_{\Gamma}w\rho d\Gamma|\leq \left(\int_{\Gamma}w^{2}d\Gamma\right)^{\frac{1}{2}}\left(\int_{\Gamma}\rho^{2}d\Gamma\right)^{\frac{1}{2}}\leq \epsilon\|w\|_{>}\epsilon\|w\|_{>}\leq \epsilon\|w\|_{>}^{2}+\frac{1}{4\epsilon}\|\rho\|_{>}^{2}.
   \end{equation*}

   \begin{equation*}
    |<w,\rho>|\leq \epsilon_{1}\|w\|_{>}^{2}+\frac{1}{4\epsilon_{1}}\|\rho\|_{>}^{2}\leq \epsilon_{1}\|w\|_{>}^{2}+\frac{1}{4\epsilon_{1}}\widehat{C}_{1}\|\rho\|_{0}
    \|\nabla\rho\|_{\bar{0}}\leq \epsilon_{1}\|w\|_{>}^{2}+\frac{1}{4\epsilon_{1}}\widehat{C}_{1}^{2}\|\nabla\rho\|_{\bar{0}}^{2}.
   \end{equation*}
   \end{proof}

     \begin{theorem} \label{t1} (Stability analysis).
      Let $(T_{h}^{n+\frac{1}{2}},\phi_{h}^{n})\in\mathcal{U}_{h}\times\mathcal{V}_{h}$, be the approximate solution provided by the first stage of the developed numerical
       technique $(\ref{s1})$-$(\ref{s3})$. Under the time step requirement $(\ref{ts})$, the following estimates hold
        \begin{equation*}
   \underset{0\leq n\leq N-1}{\max}\|T_{h}^{n+\frac{1}{2}}\|_{0}^{2} \leq \|T^{-\frac{1}{2}}\|_{0}^{2}+ \frac{T_{f}}{2}+
    \frac{T_{f}}{\tau}\left(\frac{3(\alpha\pi)^{2}}{2 c_{m}}+\frac{(\widehat{C}_{1}k_{0})^{2}(\gamma_{1}^{2}+\gamma_{2}^{2})}{3\tau^{2}|\gamma_{2}-\gamma_{1}|^{2}c_{0}}
     \min\left\{\underset{>0}{\underbrace{k_{0}-\frac{\tau\beta_{1}^{2}\mu_{2}^{2}}{6\alpha_{2}}}},\text{\,\,}\underset{>0}{\underbrace{k_{0}-\frac{\tau\beta_{2}^{2}
       \mu_{1}^{2}}{6\alpha_{1}}}}\right\}^{-1}*\right.
      \end{equation*}
      \begin{equation}\label{sc}
      \left.\underset{j=1}{\overset{2}\sum}\frac{\eta_{j}^{2}\mu_{j}^{2}}{\alpha_{j}}\right)\|f(T_{m})\|_{>}^{2}+
       \frac{8T_{f}\widehat{C}_{1}^{2}(k_{0}\pi)^{2}}{\tau^{4}c_{0}}\left(2+\frac{\widehat{C}_{1}^{2}k_{0}^{2}(\gamma_{1}^{2}+\gamma_{2}^{2})}
      {\tau^{2}|\gamma_{2}-\gamma_{1}|^{2}l_{0}}\min\left\{k_{0}-\frac{\tau\beta_{1}^{2}\mu_{2}^{2}}{6\alpha_{2}}, k_{0}-\frac{\tau\beta_{2}^{2}
\mu_{1}^{2}}{6\alpha_{1}}\right\}^{-1}\underset{j=1}{\overset{2}\sum}\frac{1}{\mu_{j}^{2}}\right)\|f(T_{h}^{n})\|_{0}^{2},
      \end{equation}
      \begin{equation}\label{sc1}
       \underset{0\leq n\leq N}{\max}\|\phi_{lh}^{n}\|_{0}^{2}\leq \min\left\{\underset{>0}{\underbrace{k_{0}-\frac{\tau\beta_{1}^{2}\mu_{2}^{2}}{6\alpha_{2}}}},
       \text{\,\,}\underset{>0}{\underbrace{k_{0}-\frac{\tau\beta_{2}^{2}\mu_{1}^{2}}{6\alpha_{1}}}}\right\}^{-1}
       \left[\frac{\tau}{6}\underset{j=1}{\overset{2}\sum}\frac{\eta_{j}^{2}\mu_{j}^{2}}{\alpha_{j}}\|f(T_{m})\|_{>}^{2}+\frac{4\pi^{2}k_{0}^{2}\widehat{C}_{1}^{2}}{\tau^{2}l_{0}}
    \underset{j=1}{\overset{2}\sum}\frac{1}{\mu_{j}^{2}}\|f(T_{h}^{n})\|_{0}^{2}\right],
     \end{equation}
     for $l=1,2,$ where $T^{-\frac{1}{2}}$ is given by approximation $(\ref{23})$ and $l_{0}$, $c_{0}>0$, are constants independent of the space step $h$ and time step
     $\sigma$, and $k_{0}>\frac{\tau}{6}\max\{\frac{\beta_{1}^{2}\mu_{2}^{2}}{\alpha_{2}},\text{\,} \frac{\beta_{2}^{2}\mu_{1}^{2}}{\alpha_{1}}\}$ (see Section $\ref{sec2}$, Page 4). It's important to recall that the term $T_{h}^{n}$ is obtained at the second stage given by equations $(\ref{s2})$, which work with implicit methods. Hence,
     $\|T_{h}^{n}\|_{0}$ is bounded, for $n=0,1,2,...,N$, with a positive constant independent of the mesh grid $h$ and time step $\sigma$.
    \end{theorem}

    \begin{proof}
     Replacing in system $(\ref{s1})$, $w$ with $T_{h}^{n}$, $\psi$ with $\phi_{h}^{n}=(\phi_{1h}^{n},\phi_{2h}^{n})^{t}$ and rearranging terms, this gives
     \begin{equation*}
    \begin{array}{c}
    \left(T_{h}^{n+\frac{1}{2}}-T_{h}^{n-\frac{1}{2}},T_{h}^{n}\right)_{0}+\sigma a(T_{h}^{n},T_{h}^{n})=\sigma\tau^{-1}<c_{m}(T_{m}-T_{h}^{n})+\alpha\pi(f(T_{m})- f(T_{h}^{n})),T_{h}^{n}>\\ -4\pi k_{0}\tau^{-2}\sigma\left(f(T_{h}^{n}),T_{h}^{n}\right)_{0}+ k_{0}\tau^{-2}(\gamma_{2}-\gamma_{1})^{-1}\sigma\left(\gamma_{2}\phi_{1h}^{n}-\gamma_{1}\phi_{2h}^{n},T_{h}^{n}\right)_{0},\\
     \text{\,}\\
    \tau^{2}\overline{\mu}^{.2}.*A(\phi_{h}^{n},\phi_{h}^{n})+k_{0}\left(\phi_{h}^{n},\phi_{h}^{n}\right)_{\Delta}+\frac{\tau}{3}\overline{\mu}^{.2}
    .*<G_{\overline{\alpha}\overline{\beta}}(\phi_{h}^{n}),\phi_{h}^{n}>_{\Delta}=\frac{\tau}{3}\overline{\mu}^{.2}.*\overline{\eta}.*<F(T_{m}),\phi_{h}^{n}>_{\Delta}+\\
     4\pi k_{0}\left(F(T_{h}^{n}),\phi_{h}^{n}\right)_{\Delta}.
    \end{array}
   \end{equation*}

   Utilizing the definition of operators $a(\cdot,\cdot)$, $\left(\cdot,\cdot\right)_{\Delta}$, $A(\cdot,\cdot)$, $<\cdot,\cdot>_{\Delta}$ and $G_{\overline{\alpha}\overline{\beta}}(\cdot)$, this system of equations is equivalent to
      \begin{equation}\label{33}
       \begin{array}{c}
         \left(T_{h}^{n+\frac{1}{2}}-T_{h}^{n-\frac{1}{2}},T_{h}^{n}\right)_{0}+\sigma\left(M_{h}^{n}\nabla T_{h}^{n},\nabla T_{h}^{n}\right)_{\bar{0}}= \sigma\tau^{-1}<c_{m}(T_{m}-T_{h}^{n})+\alpha\pi(f(T_{m})- f(T_{h}^{n})),T_{h}^{n}>\\ -4\pi k_{0}\tau^{-2}\sigma\left(f(T_{h}^{n}),T_{h}^{n}\right)_{0}+ k_{0}\tau^{-2}(\gamma_{2}-\gamma_{1})^{-1}\sigma\left(\gamma_{2}\phi_{1h}^{n}-\gamma_{1} \phi_{2h}^{n},T_{h}^{n}\right)_{0},\\
     \text{\,}\\
    \tau^{2}\mu_{1}^{2}\left(L\nabla\phi_{1h}^{n},\nabla\phi_{1h}^{n}\right)_{\bar{0}}+k_{0}\left(\phi_{1h}^{n},\phi_{1h}^{n}\right)_{0}+\frac{\tau}{3}\mu_{1}^{2}
    <\alpha_{1}\phi_{1h}^{n}+\beta_{2}\phi_{2h}^{n},\phi_{1h}^{n}>=\frac{\tau}{3}\mu_{1}^{2}\eta_{1}<f(T_{m}),\phi_{1h}^{n}>+\\
     4\pi k_{0}\left(f(T_{h}^{n}),\phi_{1h}^{n}\right)_{0},\\
    \text{\,}\\
    \tau^{2}\mu_{2}^{2}\left(L\nabla\phi_{2h}^{n},\nabla\phi_{2h}^{n}\right)_{\bar{0}}+k_{0}\left(\phi_{2h}^{n},\phi_{2h}^{n}\right)_{0}+\frac{\tau}{3}\mu_{2}^{2}
    <\alpha_{2}\phi_{2h}^{n}+\beta_{1}\phi_{1h}^{n},\phi_{2h}^{n}>=\frac{\tau}{3}\mu_{2}^{2}\eta_{2}<f(T_{m}),\phi_{2h}^{n}>+\\
     4\pi k_{0}\left(f(T_{h}^{n}),\phi_{2h}^{n}\right)_{0},
    \end{array}
   \end{equation}
    where $M_{h}^{n}=M(T_{h}^{n})$, $M$ and $L$ being the symmetric positive definite matrices given in equation $(\ref{7a})$. Since the functions $m_{ii}(\cdot)$ are continuous and bounded below over $\mathbb{R}$, so utilizing equations $(\ref{7a})$-$(\ref{7aa})$, it holds
     \begin{equation*}
      \left(M_{h}^{n}\nabla T_{h}^{n},\nabla T_{h}^{n}\right)_{\bar{0}}\geq\underset{1\leq i\leq3}{\min}\{m_{ii}(T_{h}^{n})\}\|\nabla T_{h}^{n}\|_{\bar{0}}^{2},\text{\,\,}
      \left(L\nabla \phi_{jh}^{n},\nabla \phi_{jh}^{n}\right)_{\bar{0}}\geq\frac{1}{3}\underset{1\leq i\leq3}{\min}\{(k_{0}+\sigma_{i})^{-1}\} \|\nabla\phi_{jh}^{n}\|_{\bar{0}}^{2}, \text{\,\,for\,\,}j=1,2.
     \end{equation*}

     Let $\underset{1\leq i\leq3}{\min}\{m_{ii}(T_{h}^{n})\}\geq c_{0}>0$ and $l_{0}=\frac{1}{3}\underset{1\leq i\leq3}{\min}\{(k_{0}+\sigma_{i})^{-1}\}>0$, where $c_{0}$ and $l_{0}$ are two positive constants independent of the time step $\sigma$ and mesh grid $h$. Thus,
     \begin{equation}\label{e1}
      \left(M_{h}^{n}\nabla T_{h}^{n},\nabla T_{h}^{n}\right)_{\bar{0}}\geq c_{0}\|\nabla T_{h}^{n}\|_{\bar{0}}^{2},\text{\,\,}
      \left(L\nabla \phi_{jh}^{n},\nabla \phi_{jh}^{n}\right)_{\bar{0}}\geq l_{0}\|\nabla\phi_{jh}^{n}\|_{\bar{0}}^{2}, \text{\,\,for\,\,}j=1,2.
     \end{equation}

     Using the Poincar\'{e}-Friedrich inequality $(\ref{30})$ along with estimates $(\ref{32})$ and $(\ref{e1})$, and performing straightforward computations, the last two equations in system $(\ref{33})$ imply
      \begin{equation*}
       k_{0}\|\phi_{jh}^{n}\|_{0}^{2}\leq \frac{\tau\beta_{l}^{2}\mu_{j}^{2}}{6\alpha_{j}}\|\phi_{lh}^{n}\|_{>}^{2}+\frac{\tau\eta_{j}^{2}\mu_{j}^{2}}{6\alpha_{j}}
\|f(T_{m})\|_{>}^{2}+\frac{4\pi^{2}k_{0}^{2}\widehat{C}_{1}^{2}}{\tau^{2}\mu_{j}^{2}l_{0}}\|f(T_{h}^{n})\|_{0}^{2},
     \end{equation*}
     for $j,l=1,2$, with $j\neq l$. Plugging both estimates provides
     \begin{equation*}
       k_{0}(\|\phi_{1h}^{n}\|_{0}^{2}+\|\phi_{2h}^{n}\|_{0}^{2})\leq\frac{\tau}{6}(\frac{\beta_{2}^{2}\mu_{1}^{2}}{\alpha_{1}}\|\phi_{2h}^{n}\|_{>}^{2}+
    \frac{\beta_{1}^{2}\mu_{2}^{2}}{\alpha_{2}}\|\phi_{1h}^{n}\|_{>}^{2})+\frac{\tau}{6}\underset{j=1}{\overset{2}\sum}\frac{\eta_{j}^{2}\mu_{j}^{2}}{\alpha_{j}}
\|f(T_{m})\|_{>}^{2}+\frac{4\pi^{2}k_{0}^{2}\widehat{C}_{1}^{2}}{\tau^{2}l_{0}}\underset{j=1}{\overset{2}\sum}\frac{1}{\mu_{j}^{2}}\|f(T_{h}^{n})\|_{0}^{2},
     \end{equation*}
     which is equivalent to
    \begin{equation*}
       \left(k_{0}-\frac{\tau\beta_{1}^{2}\mu_{2}^{2}}{6\alpha_{2}}\right)\|\phi_{1h}^{n}\|_{0}^{2}+\left(k_{0}-\frac{\tau\beta_{2}^{2}\mu_{1}^{2}}{6\alpha_{1}}\right)
      \|\phi_{2h}^{n}\|_{0}^{2})\leq\frac{\tau}{6}\underset{j=1}{\overset{2}\sum}\frac{\eta_{j}^{2}\mu_{j}^{2}}{\alpha_{j}}
\|f(T_{m})\|_{>}^{2}+\frac{4\pi^{2}k_{0}^{2}\widehat{C}_{1}^{2}}{\tau^{2}l_{0}}\underset{j=1}{\overset{2}\sum}\frac{1}{\mu_{j}^{2}}\|f(T_{h}^{n})\|_{0}^{2}.
     \end{equation*}

     But it's assumed in Section $\ref{sec2}$, Page 4, that $k_{0}>\frac{\tau}{6}\max\{\frac{\beta_{1}^{2}\mu_{2}^{2}}{\alpha_{2}},\text{\,}
      \frac{\beta_{2}^{2}\mu_{1}^{2}}{\alpha_{1}}\}$. Thus, this inequality implies
      \begin{equation}\label{34}
       \|\phi_{1h}^{n}\|_{0}^{2}+\|\phi_{2h}^{n}\|_{0}^{2}\leq \min\left\{k_{0}-\frac{\tau\beta_{1}^{2}\mu_{2}^{2}}{6\alpha_{2}}, k_{0}-\frac{\tau\beta_{2}^{2}\mu_{1}^{2}}{6\alpha_{1}}\right\}^{-1}\left[\frac{\tau}{6}\underset{j=1}{\overset{2}\sum}\frac{\eta_{j}^{2}\mu_{j}^{2}}{\alpha_{j}}
\|f(T_{m})\|_{>}^{2}+\frac{4\pi^{2}k_{0}^{2}\widehat{C}_{1}^{2}}{\tau^{2}l_{0}}\underset{j=1}{\overset{2}\sum}\frac{1}{\mu_{j}^{2}}\|f(T_{h}^{n})\|_{0}^{2}\right].
     \end{equation}

      Performing direct calculations, it is not difficult to show that
      \begin{equation}\label{35}
      \left(T_{h}^{n+\frac{1}{2}}-T_{h}^{n-\frac{1}{2}},T_{h}^{n}\right)_{0}=\frac{1}{2}(\|T_{h}^{n+\frac{1}{2}}\|_{0}^{2}-\|T_{h}^{n-\frac{1}{2}}\|_{0}^{2})+
      \frac{1}{2}(\|T_{h}^{n}-T_{h}^{n-\frac{1}{2}}\|_{0}^{2}-\|T_{h}^{n+\frac{1}{2}}-T_{h}^{n}\|_{0}^{2}),
     \end{equation}
     \begin{equation}\label{36}
      \sigma\left(M_{h}^{n}\nabla T_{h}^{n},\nabla T_{h}^{n}\right)_{\bar{0}}\geq \sigma c_{0}\|\nabla T_{h}^{n}\|_{\bar{0}}^{2}\geq
      \sigma\widehat{C}_{1}^{-2}c_{0}\| T_{h}^{n}\|_{0}^{2},
     \end{equation}
     \begin{equation*}
      \sigma\tau^{-1}<c_{m}(T_{m}-T_{h}^{n})+\alpha\pi(f(T_{m})-f(T_{h}^{n})),T_{h}^{n}>=-\sigma\tau^{-1}c_{m}\|T_{h}^{n}\|_{>}^{2}+\sigma\tau^{-1}c_{m}<T_{m},T_{h}^{n}>-
     \end{equation*}
      \begin{equation*}
      \sigma\tau^{-1}\alpha\pi<f(T_{h}^{n}),T_{h}^{n}>+\sigma\tau^{-1}\alpha\pi<f(T_{m}),T_{h}^{n}>\leq -\sigma\tau^{-1}c_{m}\|T_{h}^{n}\|_{>}^{2}+\sigma \epsilon\|T_{h}^{n}\|_{>}^{2}+\frac{\tau^{-2}c_{m}^{2}\sigma}{4\epsilon}\|T_{m}\|_{>}^{2}+
     \end{equation*}
     \begin{equation*}
      2\sigma\epsilon\|T_{h}^{n}\|_{>}^{2}+\frac{\sigma(\alpha\pi\tau^{-1})^{2}}{4\epsilon}(\|f(T_{h}^{n})\|_{>}^{2}+\|f(T_{m})\|_{>}^{2}).
     \end{equation*}

      For $\epsilon=\frac{1}{3}\tau^{-1}c_{m}$, using the first inequality in Lemma $\ref{l1}$, this implies
      \begin{equation}\label{37}
      \sigma\tau^{-1}<c_{m}(T_{m}-T_{h}^{n})+\alpha\pi(f(T_{m})-f(T_{h}^{n})),T_{h}^{n}>\leq \frac{3\tau^{-1}c_{m}\sigma}{4}\|T_{m}\|_{>}^{2}+ \frac{3\sigma\tau^{-1}(\alpha\pi)^{2}}{4c_{m}}(\|f(T_{h}^{n})\|_{>}^{2}+\|f(T_{m})\|_{>}^{2}),
     \end{equation}
     \begin{equation}\label{38}
      -4\pi k_{0}\tau^{-2}\sigma\left(f(T_{h}^{n}),T_{h}^{n}\right)_{0}\leq 4\pi k_{0}\tau^{-2}\sigma\|T_{h}^{n}\|_{0}\|f(T_{h}^{n})\|_{0}\leq \frac{\sigma}{2\widehat{C}_{1}^{2}}
     c_{0}\|T_{h}^{n}\|_{0}^{2}+\frac{8\sigma(\widehat{C}_{1}k_{0}\pi\tau^{-2})^{2}}{c_{0}}\|f(T_{h}^{n})\|_{0}^{2},
     \end{equation}
      \begin{equation*}
     k_{0}\tau^{-2}\sigma(\gamma_{2}-\gamma_{1})^{-1}\left(\gamma_{2}\phi_{1h}^{n}-\gamma_{1} \phi_{2h}^{n},T_{h}^{n}\right)_{0}\leq
     k_{0}\tau^{-2}\sigma|\gamma_{2}-\gamma_{1}|^{-1}\|\gamma_{2}\phi_{1h}^{n}-\gamma_{1} \phi_{2h}^{n}\|_{0}\|T_{h}^{n}\|_{0}\leq
     \end{equation*}
     \begin{equation*}
     \frac{\sigma}{2\widehat{C}_{1}^{2}}c_{0}\|T_{h}^{n}\|_{0}^{2}+\frac{\sigma\widehat{C}_{1}^{2}k_{0}^{2}}{2\tau^{4}|\gamma_{2}-\gamma_{1}|^{2}c_{0}}
      \|\gamma_{2}\phi_{1h}^{n}-\gamma_{1} \phi_{2h}^{n}\|_{0}^{2}\leq\frac{\sigma}{2\widehat{C}_{1}^{2}}c_{0}\|T_{h}^{n}\|_{0}^{2}+
     \end{equation*}
       \begin{equation}\label{39}
     \frac{\sigma\widehat{C}_{1}^{2}k_{0}^{2}(\gamma_{1}^{2}+\gamma_{2}^{2})}{\tau^{4}|\gamma_{2}-\gamma_{1}|^{2}c_{0}}(\|\phi_{1h}^{n}\|_{0}^{2}+\|\phi_{2h}^{n}\|_{0}^{2}).
     \end{equation}

     Combining estimates $(\ref{34})$ and $(\ref{39})$ and rearranging terms, it is easy to see that
       \begin{equation*}
     k_{0}\tau^{-2}\sigma(\gamma_{2}-\gamma_{1})^{-1}\left(\gamma_{2}\phi_{1h}^{n}-\gamma_{1} \phi_{2h}^{n},T_{h}^{n}\right)_{0}\leq\frac{\sigma}{2\widehat{C}_{1}^{2}} c_{0}\|T_{h}^{n}\|_{0}^{2}+\frac{\sigma\widehat{C}_{1}^{2}k_{0}^{2}(\gamma_{1}^{2}+\gamma_{2}^{2})}{\tau^{4}|\gamma_{2}-\gamma_{1}|^{2}c_{0}}*
     \end{equation*}
       \begin{equation}\label{40}
     \min\left\{k_{0}-\frac{\tau\beta_{1}^{2}\mu_{2}^{2}}{6\alpha_{2}}, k_{0}-\frac{\tau\beta_{2}^{2}\mu_{1}^{2}}{6\alpha_{1}}\right\}^{-1}
      \left[\frac{\tau}{6}\underset{j=1}{\overset{2}\sum}\frac{\eta_{j}^{2}\mu_{j}^{2}}{\alpha_{j}}\|f(T_{m})\|_{>}^{2}+ \frac{4\pi^{2}k_{0}^{2}\widehat{C}_{1}^{2}}{\tau^{2}l_{0}}\underset{j=1}{\overset{2}\sum}\frac{1}{\mu_{j}^{2}}\|f(T_{h}^{n})\|_{0}^{2}\right].
     \end{equation}

     Furthermore, applying the Taylor series expansion for the function $T_{h}$, with time step $\frac{\sigma}{2}$, this yields
      \begin{equation*}
     T_{h}^{n+\frac{1}{2}}=T_{h}^{n}+\frac{\sigma}{2}\partial_{t}T_{h}(\epsilon^{n+\frac{1}{2}})\text{\,\,and\,\,}T_{h}^{n-\frac{1}{2}}=T_{h}^{n}-\frac{\sigma}{2}
      \partial_{t}T_{h}(\epsilon^{n-\frac{1}{2}}),
     \end{equation*}
     where $\epsilon^{n+\frac{1}{2}}\in(t_{n},t_{n+\frac{1}{2}})$, $\epsilon^{n-\frac{1}{2}}\in(t_{n-\frac{1}{2}},t_{n})$, and $\partial_{t}T_{h}$ is defined in the sense of distribution. This shows that the above approximations are well defined. It follows from
     these approximations that
     \begin{equation*}
     \frac{1}{2}(\|T_{h}^{n}-T_{h}^{n-\frac{1}{2}}\|_{0}^{2}-\|T_{h}^{n+\frac{1}{2}}-T_{h}^{n}\|_{0}^{2})=\frac{\sigma^{2}}{4}(
      \|\partial_{t}T_{h}(\epsilon^{n-\frac{1}{2}})\|_{0}^{2}-\|\partial_{t}T_{h}(\epsilon^{n+\frac{1}{2}})\|_{0}^{2}).
     \end{equation*}

     Substituting this into equation $(\ref{35})$, this gives
     \begin{equation}\label{41}
      \left(T_{h}^{n+\frac{1}{2}}-T_{h}^{n-\frac{1}{2}},T_{h}^{n}\right)_{0}=\frac{1}{2}(\|T_{h}^{n+\frac{1}{2}}\|_{0}^{2}-\|T_{h}^{n-\frac{1}{2}}\|_{0}^{2})+
      \frac{\sigma^{2}}{4}(\|\partial_{t}T_{h}(\epsilon^{n-\frac{1}{2}})\|_{0}^{2}-\|\partial_{t}T_{h}(\epsilon^{n+\frac{1}{2}})\|_{0}^{2}).
     \end{equation}

       A combination of the first equation in system $(\ref{33})$ and estimates $(\ref{36})$-$(\ref{38})$ and $(\ref{40})$-$(\ref{41})$, results in
      \begin{equation*}
      \frac{1}{2}(\|T_{h}^{n+\frac{1}{2}}\|_{0}^{2}-\|T_{h}^{n-\frac{1}{2}}\|_{0}^{2})+\frac{\sigma^{2}}{4}(|\partial_{t}T_{h}(\epsilon^{n-\frac{1}{2}})\|_{0}^{2}-
      \|\partial_{t}T_{h}(\epsilon^{n+\frac{1}{2}})\|_{0}^{2})\leq \frac{3\sigma\tau^{-1}(\alpha\pi)^{2}}{4c_{m}}\|f(T_{h}^{n})\|_{>}^{2}+
      \end{equation*}
      \begin{equation*}
      \frac{\sigma}{2\tau}\left(\frac{3(\alpha\pi)^{2}}{2 c_{m}}+\frac{\widehat{C}_{1}^{2}k_{0}^{2}(\gamma_{1}^{2}+\gamma_{2}^{2})}{3\tau^{2}
       |\gamma_{2}-\gamma_{1}|^{2}c_{0}}\min\left\{k_{0}-\frac{\tau\beta_{1}^{2}\mu_{2}^{2}}{6\alpha_{2}}, k_{0}-\frac{\tau\beta_{2}^{2}\mu_{1}^{2}}{6\alpha_{1}}\right\}^{-1}
      \underset{j=1}{\overset{2}\sum}\frac{\eta_{j}^{2}\mu_{j}^{2}}{\alpha_{j}}\right)\|f(T_{m})\|_{>}^{2}+
      \end{equation*}
       \begin{equation}\label{42}
       \frac{4\sigma\widehat{C}_{1}^{2}(k_{0}\pi)^{2}}{\tau^{4}c_{0}}\left(2+\frac{\widehat{C}_{1}^{2}k_{0}^{2}(\gamma_{1}^{2}+\gamma_{2}^{2})}
      {\tau^{2}|\gamma_{2}-\gamma_{1}|^{2}l_{0}}\underset{j=1}{\overset{2}\sum}\frac{1}{\mu_{j}^{2}}
       \min\left\{k_{0}-\frac{\tau\beta_{1}^{2}\mu_{2}^{2}}{6\alpha_{2}}, k_{0}-\frac{\tau\beta_{2}^{2}\mu_{1}^{2}}{6\alpha_{1}}\right\}^{-1}\right)
      \|f(T_{h}^{n})\|_{0}^{2}.
      \end{equation}

      But, $\|\partial_{t}T_{h}\|_{0}^{2}$ belongs to $L^{1}(0,T_{f})$, because the partial derivative $\partial_{t}T_{h}$ is defined in the sense of distribution. As a result, there is a positive constant $\nu_{0}$ independent of the time step $\sigma$ and space step $h$, so that $\left|\|\partial_{t}T_{h}(\epsilon^{n-\frac{1}{2}})\|_{0}^{2}-
     \|\partial_{t}T_{h}(\epsilon^{n+\frac{1}{2}})\|_{0}^{2})\right|\leq\nu_{0}$. For sufficiently small values of the time step $\sigma$, it follows from the time step restriction $(\ref{ts})$ that $\sigma\|f(T_{h}^{n})\|_{>}^{2}\approx0$. Using this fact, estimate $(\ref{42})$ becomes
       \begin{equation*}
      \|T_{h}^{n+\frac{1}{2}}\|_{0}^{2}-\|T_{h}^{n-\frac{1}{2}}\|_{0}^{2} \leq \frac{\sigma^{2}}{2}\left|\|\partial_{t}T_{h}(\epsilon^{n-\frac{1}{2}})\|_{0}^{2}-
     \|\partial_{t}T_{h}(\epsilon^{n+\frac{1}{2}})\|_{0}^{2})\right|+
      \end{equation*}
      \begin{equation*}
      \frac{\sigma}{\tau}\left(\frac{3(\alpha\pi)^{2}}{2 c_{m}}+\frac{\widehat{C}_{1}^{2}k_{0}^{2}(\gamma_{1}^{2}+\gamma_{2}^{2})}{3\tau^{2}
       |\gamma_{2}-\gamma_{1}|^{2}c_{0}}\min\left\{k_{0}-\frac{\tau\beta_{1}^{2}\mu_{2}^{2}}{6\alpha_{2}}, k_{0}-\frac{\tau\beta_{2}^{2}\mu_{1}^{2}}{6\alpha_{1}}\right\}^{-1}
      \underset{j=1}{\overset{2}\sum}\frac{\eta_{j}^{2}\mu_{j}^{2}}{\alpha_{j}}\right)\|f(T_{m})\|_{>}^{2}+
      \end{equation*}
       \begin{equation*}
       \frac{8\sigma\widehat{C}_{1}^{2}(k_{0}\pi)^{2}}{\tau^{4}c_{0}}\left(2+\frac{\widehat{C}_{1}^{2}k_{0}^{2}(\gamma_{1}^{2}+\gamma_{2}^{2})}
      {\tau^{2}|\gamma_{2}-\gamma_{1}|^{2}l_{0}}\underset{j=1}{\overset{2}\sum}\frac{1}{\mu_{j}^{2}}
       \min\left\{k_{0}-\frac{\tau\beta_{1}^{2}\mu_{2}^{2}}{6\alpha_{2}}, k_{0}-\frac{\tau\beta_{2}^{2}\mu_{1}^{2}}{6\alpha_{1}}\right\}^{-1}\right)
      \|f(T_{h}^{n})\|_{0}^{2}.
      \end{equation*}

     Summing up this estimate for $l=0,1,...,n$, and rearranging term, this implies
       \begin{equation*}
      \|T_{h}^{n+\frac{1}{2}}\|_{0}^{2}\leq \frac{\sigma\nu_{0}T_{f}}{2}+\|T_{h}^{-\frac{1}{2}}\|_{0}^{2}+\frac{T_{f}}{\tau}\left(\frac{3(\alpha\pi)^{2}}{2 c_{m}}+\frac{\widehat{C}_{1}^{2}k_{0}^{2}(\gamma_{1}^{2}+\gamma_{2}^{2})}{3\tau^{2}|\gamma_{2}-\gamma_{1}|^{2}c_{0}}*\right.
      \end{equation*}
      \begin{equation*}
      \left.\min\left\{k_{0}-\frac{\tau\beta_{1}^{2}\mu_{2}^{2}}{6\alpha_{2}}, k_{0}-\frac{\tau\beta_{2}^{2}\mu_{1}^{2}}{6\alpha_{1}}\right\}^{-1}
      \underset{j=1}{\overset{2}\sum}\frac{\eta_{j}^{2}\mu_{j}^{2}}{\alpha_{j}}\right)\|f(T_{m})\|_{>}^{2}+
      \end{equation*}
       \begin{equation*}
       \frac{8T_{f}\widehat{C}_{1}^{2}(k_{0}\pi)^{2}}{\tau^{4}c_{0}}\left(2+\frac{\widehat{C}_{1}^{2}k_{0}^{2}(\gamma_{1}^{2}+\gamma_{2}^{2})}
      {\tau^{2}|\gamma_{2}-\gamma_{1}|^{2}l_{0}}\min\left\{k_{0}-\frac{\tau\beta_{1}^{2}\mu_{2}^{2}}{6\alpha_{2}}, k_{0}-\frac{\tau\beta_{2}^{2}
\mu_{1}^{2}}{6\alpha_{1}}\right\}^{-1}\underset{j=1}{\overset{2}\sum}\frac{1}{\mu_{j}^{2}}\right)\|f(T_{h}^{n})\|_{0}^{2},
      \end{equation*}
      for $n=0,1,...,N-1$. Indeed: $(n+1)\sigma\leq T_{f}$. Since $\sigma$ is too small, so $\sigma\nu_{0}\leq1$. Hence, the proof of Theorem $\ref{t1}$ is completed thanks to the initial condition $(\ref{s3})$ and estimate $(\ref{34})$.
      \end{proof}

      \section{Numerical experiments}\label{sec5}

      \text{\,\,\,\,\,\,\,\,\,\,}In this section we use the proposed high-order two-stage explicit/implicit approach combined with the mixed FEM $( \mathcal{P}_{4}/\mathcal{P}_{3})$ to simulate the $SP_{3}$ equations $(\ref{4})$ subjects to initial condition $(\ref{5})$ and boundary condition $(\ref{6})$. Three examples are carried out to confirm the theory and to demonstrate the utility and performance of the new algorithm $(\ref{s1})$-$(\ref{s3})$. To verify the stability and convergence order of the developed computational technique, we take $h=\frac{1}{2^{m}},$ for $m=2,3,4,5$, where $h=\max\{d(K),\text{\,\,}K\in\mathcal{F}_{h}\}$. Additionally, a uniform time step $\sigma=10^{-l}$, for $m=3,4,5,6$, is used. The errors: $e_{T,\cdot}=T_{h}^{n}-T^{n}$ and $e_{\phi,\cdot}=\phi_{h}^{n}-\phi^{n}$, along with the temperature $T_{h}$ and the term $\phi_{h}$ are computed at time $t_{n}$ using the norms $\||\cdot|\|_{0,\infty}$ and $\||\cdot|\|_{\bar{0},\infty}$, respectively, defined as
      \begin{equation*}
       \||w|\|_{0,\infty}=\underset{0\leq n\leq N}{\max}\|w^{n}\|_{0},\text{\,\,\,}\forall w\in L^{\infty}(0,T_{f};\text{\,}L^{2}(\Omega))\text{\,\,\,and\,\,\,}\||\psi|\|_{\bar{0},\infty}=\underset{0\leq n\leq N}{\max}\|\psi^{n}\|_{\bar{0}},\text{\,\,\,}\forall \psi\in[L^{\infty}(0,T_{f};\text{\,}L^{2}(\Omega))]^{2}.
         \end{equation*}

        In addition, the space convergence order, $CO(h)$, of the constructed computational scheme is calculated utilizing the formula
         \begin{equation*}
          CO(h)=\frac{\log\left(\frac{\||w_{2h}-w|\|_{0,\infty}}{\||w_{h}-w|\|_{0,\infty}}\right)}{\log(2)},\text{\,\,}
          \frac{\log\left(\frac{\||\psi_{2h}-\psi|\|_{\bar{0},\infty}}{\||\psi_{h}-\psi|\|_{\bar{0},\infty}}\right)}{\log(2)},
         \end{equation*}
          where, $v_{2h}$ and $v_{h}$ are the spatial errors associated with the grid sizes $2h$ and $h$, respectively, while the temporal convergence order, $CO(\sigma)$, is estimated using the formula
         \begin{equation*}
          CO(\sigma)=\frac{\log\left(\frac{\||w_{10\sigma}-w|\|_{0,\infty}}{\||w_{\sigma}-w|\|_{0,\infty}}\right)}{\log(10)},\text{\,\,}
          \frac{\log\left(\frac{\||\psi_{10\sigma}-\psi|\|_{\bar{0},\infty}}{\||\psi_{\sigma}-\psi|\|_{\bar{0},\infty}}\right)}{\log(10)},
         \end{equation*}
         where $z_{\sigma}$ and $z_{10\sigma}$ represent the temporal errors corresponding to time steps $\sigma$ and $10\sigma$, respectively. The numerical calculations are performed using MATLAB R$2007b$.\\

          $\bullet$ \textbf{Example 1}. We consider the $SP_{3}$ problem $(\ref{4})$ defined on the domain $\overline{\Omega}=[0,\text{\,}1]^{3}$. The final time $T_{f}=0.1$. The parameters are given by: $\alpha=0$, $\tau=1$, $k_{0}=1$, $\sigma_{j}=0$, for $j=1,2,3$. The tensor $M=I_{3}$, where $I_{3}$ designates the identity matrix of size $3$, while the other parameters are provided in \textbf{Table 1}. A solution of the initial-boundary problem $(\ref{4})$-$(\ref{6})$ is constructed by adding extra source functions in the right-hand side of equations $(\ref{4})$, such that the exact solutions are given by
     \begin{equation}\label{exact}
       \left.
         \begin{array}{ll}
           T(x,y,z,t)=\sin(2\pi x)\sin(2\pi y)\sin(2\pi z)e^{t}, & \hbox{} \\
           \phi(x,y,z,t)=(\phi_{1}(x,y,z,t),\phi_{2}(x,y,z,t))^{t}=(\tanh(T+1),\tanh(T-1))^{t}. & \hbox{}
         \end{array}
       \right.
     \end{equation}

     Additionally, the initial temperature and boundary conditions are directly obtained from the analytical solutions defined in equations $(\ref{exact})$.\\

         \textbf{Table 2.} $\label{T2}$ Convergence order $CO(h)$ of the developed high-order two-stage explicit/implicit computational approach combined with mixed FEM ($\mathcal{P}_{4}/\mathcal{P}_{3}$) for solving the $SP_{3}$ model $(\ref{s1})$-$(\ref{s3})$ using various mesh grids $h$ and a time step $\sigma=10^{-4}$.
          \begin{equation*}
          \small{\begin{array}{c c}
          \text{\,developed computational technique,\,\,where\,\,}\sigma=10^{-4}& \\
           \begin{tabular}{cccccccccc}
            \hline
            $h$      & $\||T\|_{0,\infty}$ &$\||T_{h}|\|_{0,\infty}$ & $\||\phi\|_{\bar{0},\infty}$ & $\||\phi_{h}|\|_{\bar{0},\infty}$ & $\||T_{h}-T|\|_{0,\infty}$ & $CO(h)$ & $\||\phi_{h}-\phi|\|_{\bar{0},\infty}$ & $CO(h)$ & CPU(s)\\
             \hline
            $2^{-2}$ & $0.7817$ & $0.7820$ & $2.9932$ & $2.9933$ & $2.3925\times10^{-3}$ & ....   &  $1.8729\times10^{-3}$ & .... & 27.6893\\

            $2^{-3}$ & $0.7818$ & $0.7820$ & $2.9932$ & $2.9933$ & $1.6166\times10^{-4}$ & 3.8875 &  $1.2601\times10^{-4}$ & 3.8937 & 52.5349\\

            $2^{-4}$ & $0.7819$ & $0.7821$ & $2.9933$ & $2.9933$ & $1.0226\times10^{-5}$ & 3.9827 &  $7.8299\times10^{-6}$ & 4.0084 & 111.2006\\

            $2^{-5}$ & $0.7820$ & $0.7821$ & $2.9933$ & $2.9933$ & $6.3753\times10^{-7}$ & 4.0036 &  $4.9274\times10^{-7}$ & 3.9901 & 236.2858\\
            \hline
          \end{tabular} &
          \end{array}}
          \end{equation*}
           \text{\,}\\
           \textbf{Table 3.} $\label{T3}$ Convergence order $CO(\sigma)$ of the proposed approach $(\ref{s1})$-$(\ref{s3})$ with space step $h=2^{-3}$ and different time steps $\sigma$.
           \begin{equation*}
          \small{\begin{array}{c c}
          \text{\,developed computational approach\,\,where\,\,}h=2^{-3}& \\
           \begin{tabular}{cccccccccc}
            \hline
            $\sigma$ & $\||T\|_{0,\infty}$ &$\||T_{h}|\|_{0,\infty}$ & $\||\phi\|_{\bar{0},\infty}$ & $\||\phi_{h}|\|_{\bar{0},\infty}$ & $\||T_{h}-T|\|_{0,\infty}$ & $CO(\sigma)$ & $\||\phi_{h}-\phi|\|_{\bar{0},\infty}$ & $CO(\sigma)$ & CPU(s)\\
             \hline
            $10^{-3}$ & $0.7818$ & $0.7819$ & $2.9932$ & $2.9932$ &  $1.8276\times10^{-2}$ & ....   &  $2.1879\times10^{-2}$ & .... & 26.3108\\

            $10^{-4}$ & $0.7819$ & $0.7820$ & $2.9932$ & $2.9933$ &  $2.120\times10^{-4}$ & 1.9356 &  $2.0330\times10^{-4}$ & 2.0319 & 54.8962\\

            $10^{-5}$ & $0.7819$ & $0.7820$ & $2.9933$ & $2.9933$ & $2.1293\times10^{-6}$ & 1.9981 &  $2.0325\times10^{-6}$ & 2.0001 & 129.5072\\

            $10^{-6}$ & $0.7820$ & $0.7820$ & $2.9933$ & $2.9933$ & $2.1112\times10^{-8}$ & 2.0037 &  $2.0409\times10^{-8}$ & 1.9982 & 361.2085\\
            \hline
          \end{tabular} &
          \end{array}}
          \end{equation*}

            \textbf{Table 4.} $\label{T4}$ (method discussed in \cite{bs}). $\label{T3}$ Errors and computational costs obtained using the unified and mixed discretizations with the
           considered finite elements for the accuracy example, time step $\sigma=10^{-3}$.
           \begin{equation*}
          \begin{tabular}{ccc}
            \hline
            Unified discretization &  &  \\
            \hline
            FEM  & $L^{2}$-norm & CPU(s) \\
            \hline
            $P_{1}/P_{1}/P_{1}$ & $7.6426\times10^{-3}$ & 77 \\
            $P_{2}/P_{2}/P_{2}$ & $8.1175\times10^{-5}$ & 219 \\
            $P_{3}/P_{3}/P_{3}$ & $2.5152\times10^{-6}$ & 1428 \\
            \hline
            Mixed FEM & $L^{2}$-norm & CPU(s)  \\
            \hline
            $P_{2}/P_{1}/P_{1}$ & $8.9647\times10^{-5}$ &  116  \\
            $P_{3}/P_{2}/P_{2}$ & $2.5164\times10^{-6}$ &  482 \\
            \hline
          \end{tabular}
          \end{equation*}

          $\bullet$ \textbf{Example 2}. Consider the benchmark model taken in \cite{bs} for modeling cooling materials such as glass. The domain of fluid is given by $\overline{\Omega}=[0,\text{\,}1]\times[0,\text{\,}1]\times[0,\text{\,}1]$, while the time interval is $[0,\text{\,}1]$. The physical parameters are given by: $\alpha=0$, $\tau=1$, $k_{0}=1$, $c_{m}=1$, $\sigma_{j}=0$, for $j=1,2,3$. The tensor $M=I_{3}$ while the other parameters are provided in \textbf{Table 1}. The initial temperature is $T_{0}=1500K$ whereas the boundary temperature is $T_{m}=300K$\\

          \textbf{Table 5.} $\label{T5}$ Convergence order $CO(h)$ of the proposed approach $(\ref{s1})$-$(\ref{s3})$ with space step $\sigma=2^{-6}$ and different time
          steps $\sigma$.
           \begin{equation*}
          \small{\begin{array}{c c}
          \text{\,proposed numerical technique\,\,where\,\,}\sigma=2^{-6}& \\
           \begin{tabular}{cccccccccc}
            \hline
            $h$ & $\||T\|_{0,\infty}$ &$\||T_{h}|\|_{0,\infty}$ & $\||\phi\|_{\bar{0},\infty}$ & $\||\phi_{h}|\|_{\bar{0},\infty}$ & $\||T_{h}-T|\|_{0,\infty}$ & $CO(h)$ &
             $\||\phi_{h}-\phi|\|_{\bar{0},\infty}$ & $CO(h)$ & CPU(s)\\
             \hline
            $2^{-2}$ & $3549.8$ & $3549.9$ & $4.8027$ & $4.8025$ &  $3.5792\times10^{-1}$ & ....   &  $4.8906\times10^{-2}$ & ....  & 39.3913\\

            $2^{-3}$ & $3549.8$ & $3549.8$ & $4.8027$ & $4.8026$ &  $2.5727\times10^{-2}$ & 3.7983 &  $3.4332\times10^{-3}$ & 3.8324 & 88.9810\\

            $2^{-4}$ & $3549.9$ & $3549.8$ & $4.8028$ & $4.8027$ & $1.7449\times10^{-3}$ & 3.8821 &  $2.3014\times10^{-4}$ & 3.8990 & 221.5121\\

            $2^{-5}$ & $3549.9$ & $3549.9$ & $4.8029$ & $4.8028$ & $1.1085\times10^{-4}$ & 3.9765 &  $1.4557\times10^{-5}$ & 3.9827 & 606.1729\\
            \hline
          \end{tabular} &
          \end{array}}
          \end{equation*}

         \textbf{Table 6.} $\label{T6}$ Convergence order $CO(\sigma)$ of the proposed approach $(\ref{s1})$-$(\ref{s3})$ with space step $h=2^{-4}$ and different time
         steps $\sigma$.
           \begin{equation*}
          \small{\begin{array}{c c}
          \text{\,proposed numerical technique\,\,where\,\,}h=2^{-4}& \\
           \begin{tabular}{cccccccccc}
            \hline
            $\sigma$ & $\||T\|_{0,\infty}$ &$\||T_{h}|\|_{0,\infty}$ & $\||\phi\|_{\bar{0},\infty}$ & $\||\phi_{h}|\|_{\bar{0},\infty}$ & $\||T_{h}-T|\|_{0,\infty}$ &
            $CO(\sigma)$ & $\||\phi_{h}-\phi|\|_{\bar{0},\infty}$ & $CO(\sigma)$ & CPU(s)\\
             \hline
            $10^{-3}$ & $3549.7$ & $3549.8$ & $4.8027$ & $4.8026$ &  $8.2193\times10^{-2}$ & ....   &  $5.4678\times10^{-2}$ & .... & 41.7259\\

            $10^{-4}$ & $3549.8$ & $3549.9$ & $4.8026$ & $4.8027$ &  $9.0685\times10^{-4}$ & 1.9573 &  $5.6626\times10^{-4}$ & 1.9848 & 90.0570\\

            $10^{-5}$ & $3549.8$ & $3549.9$ & $4.8027$ & $4.8027$ & $9.4631\times10^{-6}$ & 1.9815 &  $5.6887\times10^{-6}$ & 1.9980 & 224.5842\\

            $10^{-6}$ & $3549.8$ & $3549.8$ & $4.8028$ & $4.8029$ & $9.1870\times10^{-8}$ & 2.0124 &  $5.6068\times10^{-8}$ & 2.0063 & 654.9099\\
            \hline
          \end{tabular} &
          \end{array}}
          \end{equation*}

          $\bullet$ \textbf{Example 3}. We consider the $SP_{3}$ problem $(\ref{4})$ defined on the domain $\overline{\Omega}=[0,\text{\,}10]^{2}\times[-1,\text{\,}1]$. The final time $T_{f}=1$. The parameters are given by: $\alpha=10^{-2}$, $\tau=1$, $k_{0}=1$, $\sigma_{1}=\sigma_{3}=0$, $\sigma_{2}=10^{-1}$, while the other parameters are provided in \textbf{Table 1}. The thermal conductivity $M$ is nonlinear and anisotropic and it is given by equation $(\ref{7a})$-$(\ref{7aa})$ with
          \begin{equation*}
           m_{11}(T)=m_{33}(T)=5\times10^{-4}T^{2}+2\times10^{-2}T+10^{-1},\text{\,\,\,}m_{22}(T)=2\times10^{-2}T+10^{-1}\text{\,\,\,and\,\,\,}\theta=\frac{\pi}{4}.
          \end{equation*}
           The initial temperature is $T_{0}=1000K$ while the boundary temperature is $T_{m}=300K$.\\

          \textbf{Table 7.} $\label{T7}$ Convergence order $CO(h)$ of the new algorithm $(\ref{s1})$-$(\ref{s3})$ with space step $\sigma=2^{-5}$ and different space
          steps $h$.
           \begin{equation*}
          \small{\begin{array}{c c}
          \text{\,new algorithm\,\,where\,\,}\sigma=2^{-5}& \\
           \begin{tabular}{cccccccccc}
            \hline
            $h$ & $\||T\|_{0,\infty}$ &$\||T_{h}|\|_{0,\infty}$ & $\||\phi\|_{\bar{0},\infty}$ & $\||\phi_{h}|\|_{\bar{0},\infty}$ & $\||T_{h}-T|\|_{0,\infty}$ & $CO(h)$ &
             $\||\phi_{h}-\phi|\|_{\bar{0},\infty}$ & $CO(h)$ & CPU(s)\\
             \hline
            $2^{-2}$ & $24294.00$ & $24294.01$ & $48.0021$ & $48.0023$ &  $3.5792\times10^{-2}$ & ....   &  $4.8906\times10^{-3}$ & .... & 25.4183\\

            $2^{-3}$ & $24294.01$ & $24294.02$ & $48.0021$ & $48.0022$ &  $2.4188\times10^{-3}$ & 3.8873 &  $3.0926\times10^{-4}$ & 3.9831 & 54.5553\\

            $2^{-4}$ & $24294.02$ & $24294.02$ & $48.0022$ & $48.0023$ & $1.5081\times10^{-4}$ &  4.0035 &  $1.9072\times10^{-5}$ & 4.0193 & 134.266\\

            $2^{-5}$ & $24294.02$ & $24294.03$ & $48.0022$ & $48.0024$ & $9.4400\times10^{-6}$ &  3.9978 &  $1.1914\times10^{-6}$ & 4.0007 & 553.3500\\
            \hline
          \end{tabular} &
          \end{array}}
          \end{equation*}

         \textbf{Table 8.} $\label{T8}$ Convergence order $CO(\sigma)$ of the new algorithm $(\ref{s1})$-$(\ref{s3})$ with space step $h=2^{-4}$ and different time
         steps $\sigma$.
           \begin{equation*}
          \small{\begin{array}{c c}
          \text{\,new algorithm\,\,where\,\,}h=2^{-4}& \\
           \begin{tabular}{cccccccccc}
            \hline
            $\sigma$ & $\||T\|_{0,\infty}$ &$\||T_{h}|\|_{0,\infty}$ & $\||\phi\|_{\bar{0},\infty}$ & $\||\phi_{h}|\|_{\bar{0},\infty}$ & $\||T_{h}-T|\|_{0,\infty}$ &
            $CO(\sigma)$ & $\||\phi_{h}-\phi|\|_{\bar{0},\infty}$ & $CO(\sigma)$ & CPU(s)\\
             \hline
            $10^{-3}$ & $24294.01$ & $24294.00$ & $48.0022$ & $48.0021$ &  $6.0021\times10^{-3}$ & ....   &  $1.2151\times10^{-3}$ & .... & 32.6123\\

            $10^{-4}$ & $24294.01$ & $24294.01$ & $48.0022$ & $48.0021$ &  $7.5981\times10^{-5}$ & 1.8976 &  $1.2468\times10^{-5}$ & 1.9888 & 69.9818\\

            $10^{-5}$ & $24294.02$ & $24294.01$ & $48.0023$ & $48.0022$ & $1.0179\times10^{-6}$ & 1.9730 &  $1.2552\times10^{-7}$ & 1.9971 & 171.9375\\

            $10^{-6}$ & $24294.02$ & $24294.02$ & $48.0021$ & $48.0023$ & $9.6918\times10^{-9}$ & 2.0213 &  $1.2549\times10^{-9}$ & 2.0001 & 662.8750\\
            \hline
          \end{tabular} &
          \end{array}}
          \end{equation*}

          \textbf{Tables 2-3 $\&$ 5-8} suggest that the constructed two-stage explicit/implicit computational technique combined with the mixed FEM ($\mathcal{P}_{4}/\mathcal{P}_{3}$) given by equations $(\ref{s1})$-$(\ref{s3})$, for solving the initial-boundary value problem $(\ref{4})$-$(\ref{6})$, is temporal second-order accurate and spatial fourth-order convergent, while Figures $\ref{fig1}$-$\ref{fig3}$ show that the new algorithm is stable for small values of the time step $\sigma$. In addition, it follows from both tables and figures that the approximate solutions do not increase with time and converge to the analytical one.\\

         Furthermore, both computed temperature and $\phi_{jh}$, for $j=1,2$, are displayed in Figures $\ref{fig1}$-$\ref{fig3}$. A time step $\sigma\in\{10^{-l},\text{\,l=3,4,5,6}\}$, and space step $h=2^{-m},$ for $m=2,3,4,5$, are used. The figures provide the approximate temperature and functions $\phi_{jh}$, $j=1,2$, at time $t=1.5\times10^{-3}$, $5.5\times10^{-5}$ and $2\times10^{-3}$. Additionally, they indicate that the approximate solutions propagate with almost a perfectly value at different positions. As a result, the numerical solutions do not increase with time. Moreover, they show that the developed method $(\ref{s1})$-$(\ref{s3})$ is not unconditionally unstable, but stability depends on the parameter $\sigma$.

       \section{General conclusions and future works}\label{sec6}
        This paper has constructed a strong two-stage explicit/implicit numerical approach combined with the mixed FEM ($\mathcal{P}_{p}/\mathcal{P}_{p-1}$), in an approximate solution of the nonlinear $SP_{3}$ approximations $(\ref{4})$-$(\ref{6})$, associated with the three-dimensional nonlinear radiation-conduction problems in anisotropic media $(\ref{1})$-$(\ref{3})$. Under a necessary restriction on the time step given by estimate $(\ref{ts})$, the stability of the new algorithm has been deeply analyzed in the $L^{2}$-norm. The theoretical studies have indicated that the proposed computational technique is stable and temporal second-order convergent for values of the time step small enough. These theoretical results have been confirmed by some numerical examples. Specifically, Figures $\ref{fig1}$-$\ref{fig3}$ suggest that the proposed numerical approach $(\ref{s1})$-$(\ref{s3})$ is stable while \textbf{Tables 2-3 $\&$ 5-8} show that the developed technique is temporal second-order accurate and fourth-order convergent in space. These results show that the proposed approach is more efficient than a wide set of numerical methods discussed in the literature \cite{bs,43bs} for solving the $SP_{N}$ approximations. Our future works will analyze a suitable time step restriction for stability together with the error estimates of the proposed strong two-stage explicit/implicit approach combined with the mixed FEM for solving the $SP_{3}$ model.

      \subsection*{Ethical Approval}
     Not applicable.
     \subsection*{Availability of supporting data}
     Not applicable.
     \subsection*{Declaration of Interest Statement}
     The author declares that he has no conflict of interests.
     \subsection*{Funding}
     No applicable
     \subsection*{Authors' contributions}
     The whole work has been carried out by the author.

     \begin{figure}
         \begin{center}
         Stability analysis of the proposed computational approach for nonlinear radiation-conduction model, with $\sigma=10^{-4}$ and $h=2^{-2}$.
         \begin{tabular}{c c}
         \psfig{file=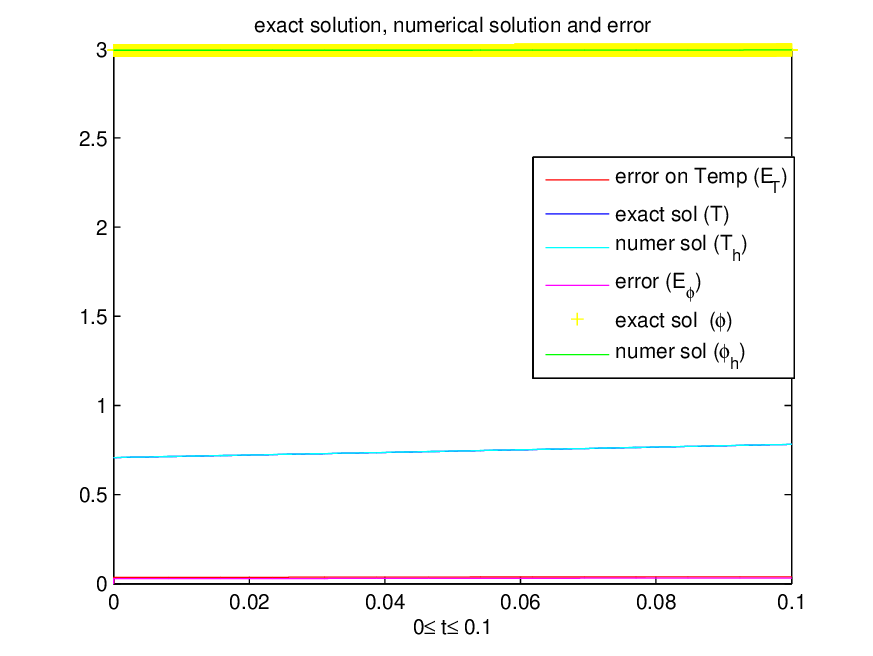,width=7cm} & \psfig{file=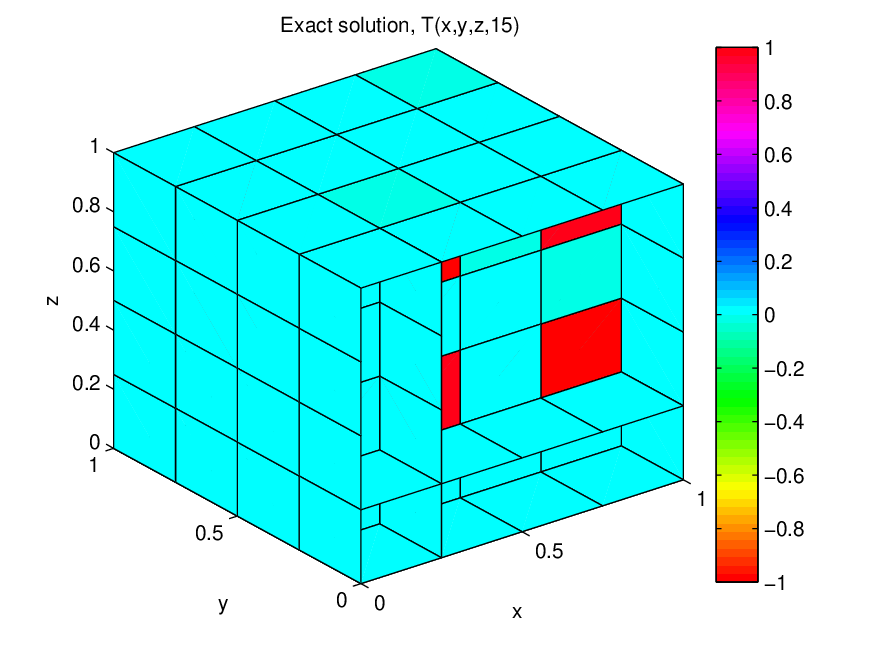,width=7cm}\\
         \psfig{file=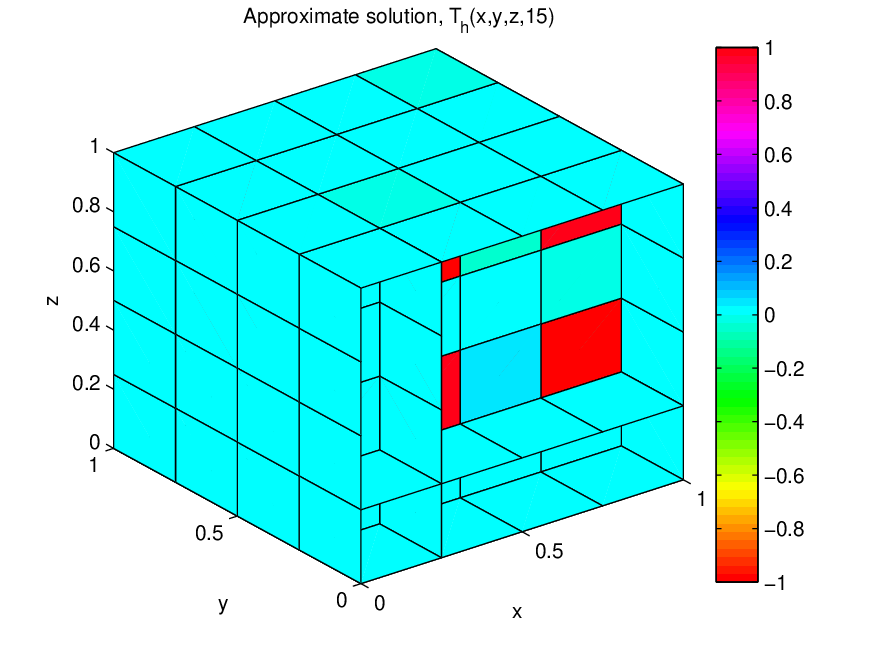,width=7cm} & \psfig{file=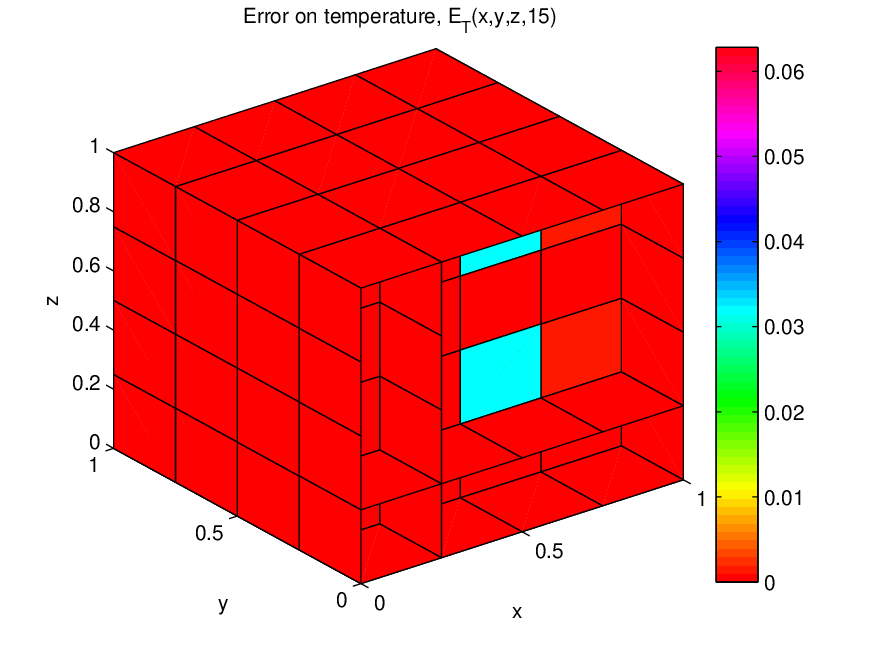,width=7cm}\\
         \psfig{file=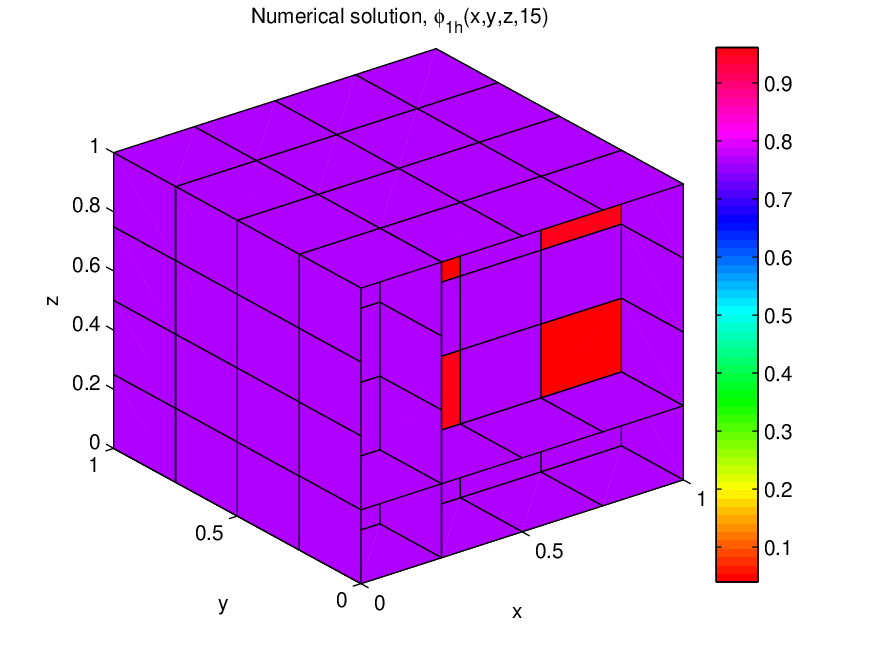,width=7cm} & \psfig{file=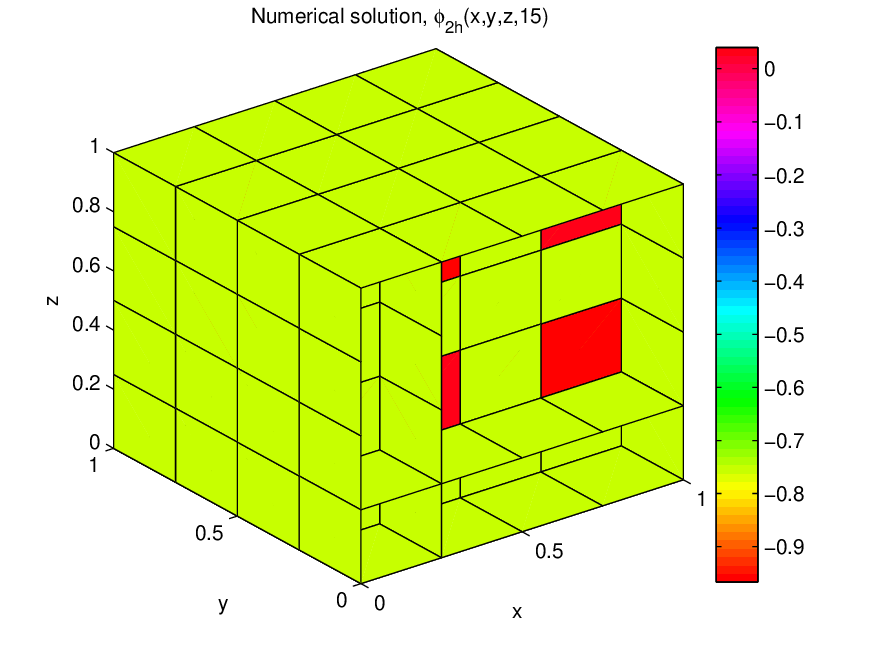,width=7cm}\\
         \psfig{file=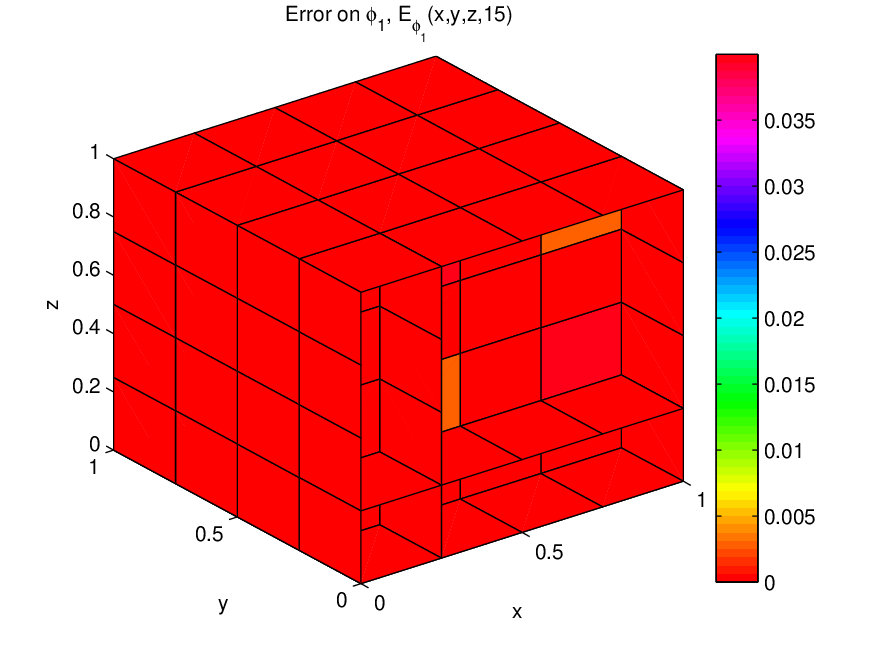,width=7cm} & \psfig{file=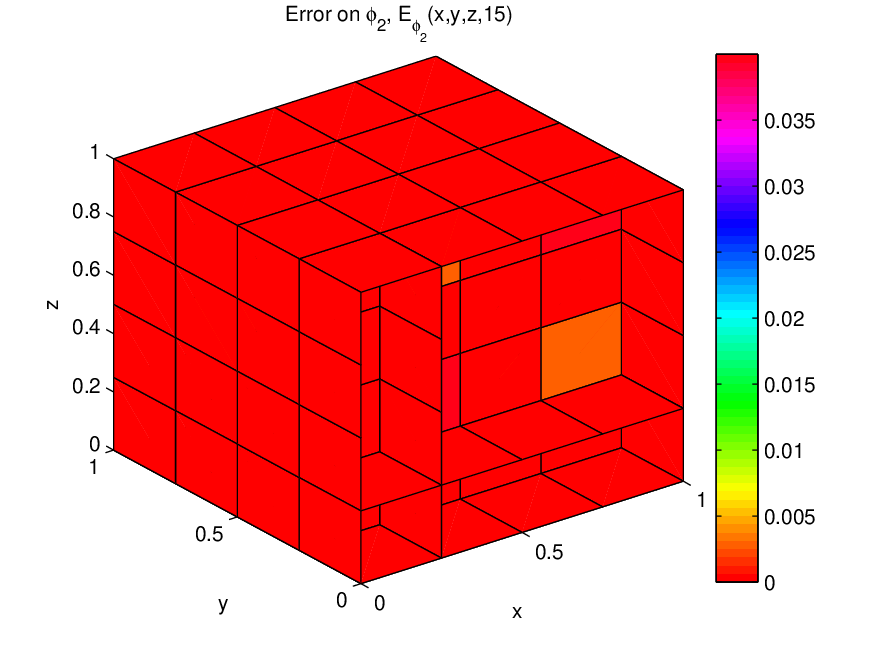,width=7cm}
         \end{tabular}
        \end{center}
        \caption{Graphs of exact and computed solutions for Temperature $(T\text{\,and\,}T_{h})$, related approximate radiation intensity $(\phi_{1h}\text{\,and\,}\phi_{2h})$
        and errors $(E_{T}\text{\,and\,}E_{\phi_{j}})$, associated with Example 1.}
        \label{fig1}
        \end{figure}

       \begin{figure}
       \begin{center}
       Stability analysis of the new algorithm for nonlinear radiation-conduction model, with $\sigma=10^{-6}$, $h=2^{-3}$.
       \begin{tabular}{c c}
         \psfig{file=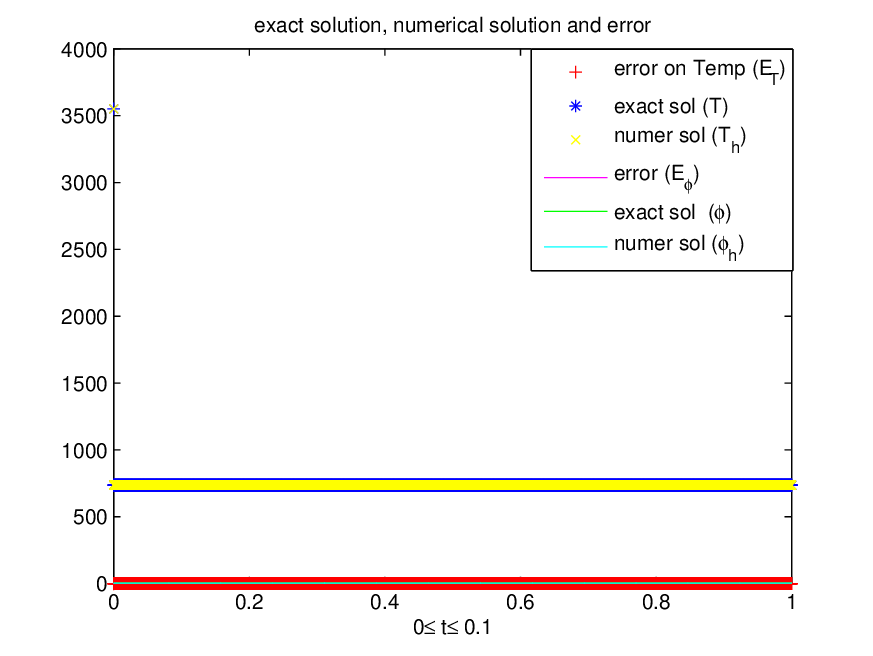,width=7cm} & \psfig{file=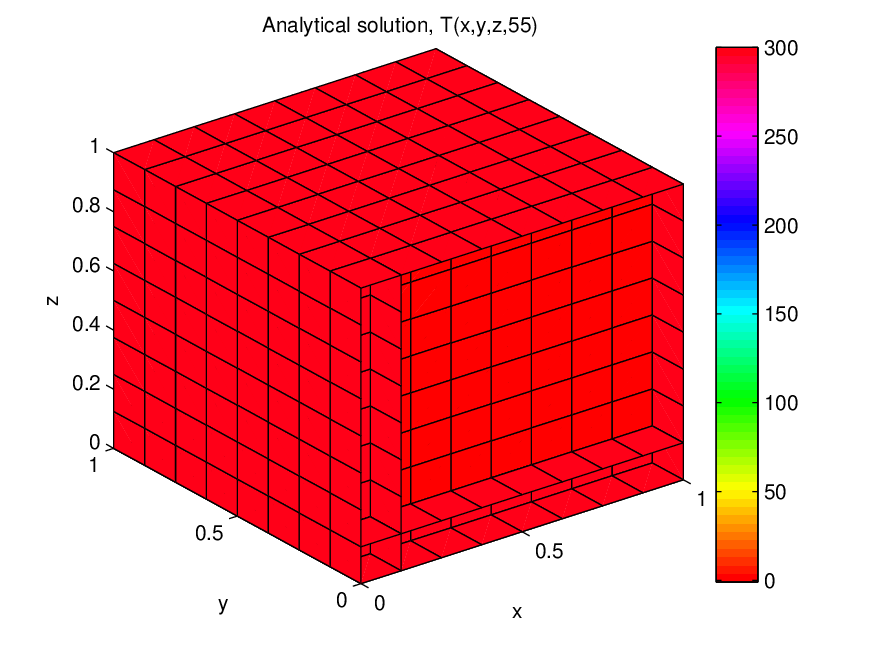,width=7cm}\\
         \psfig{file=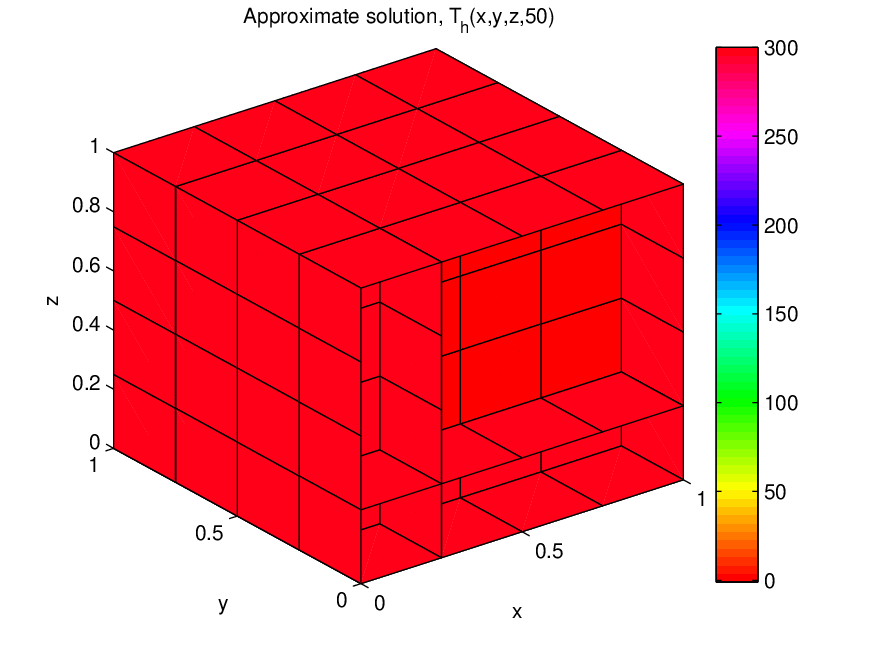,width=7cm} & \psfig{file=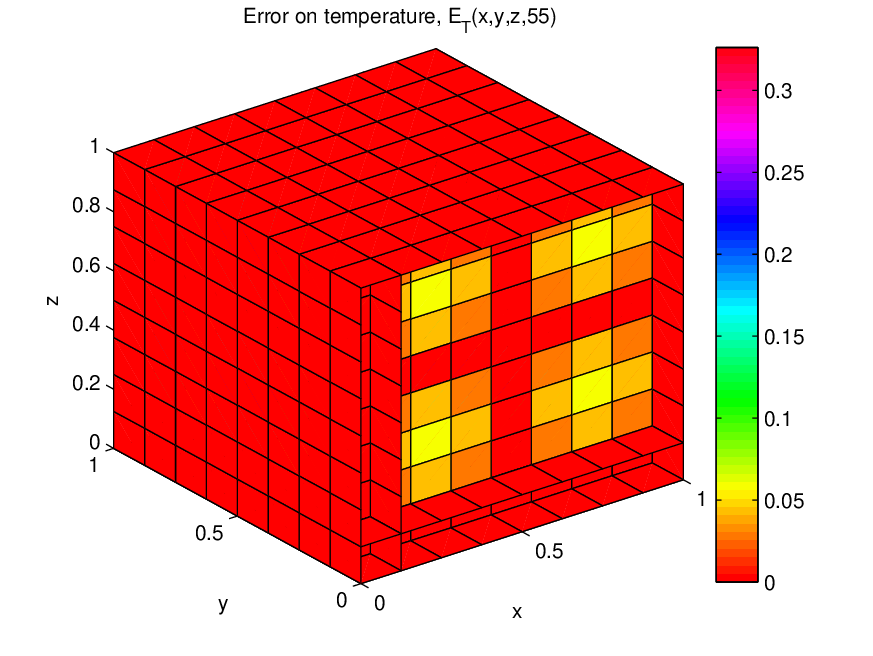,width=7cm}\\
         \psfig{file=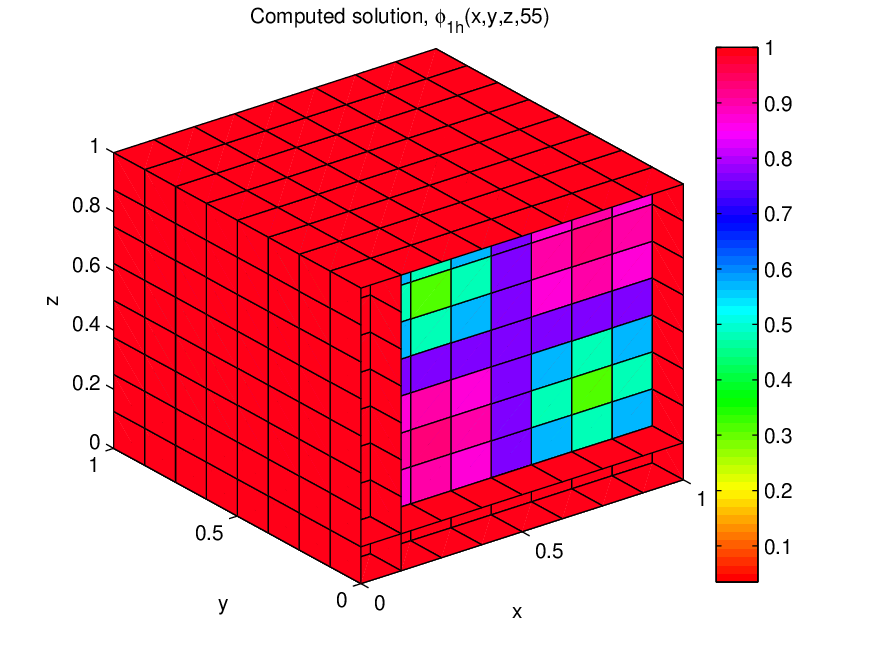,width=7cm} & \psfig{file=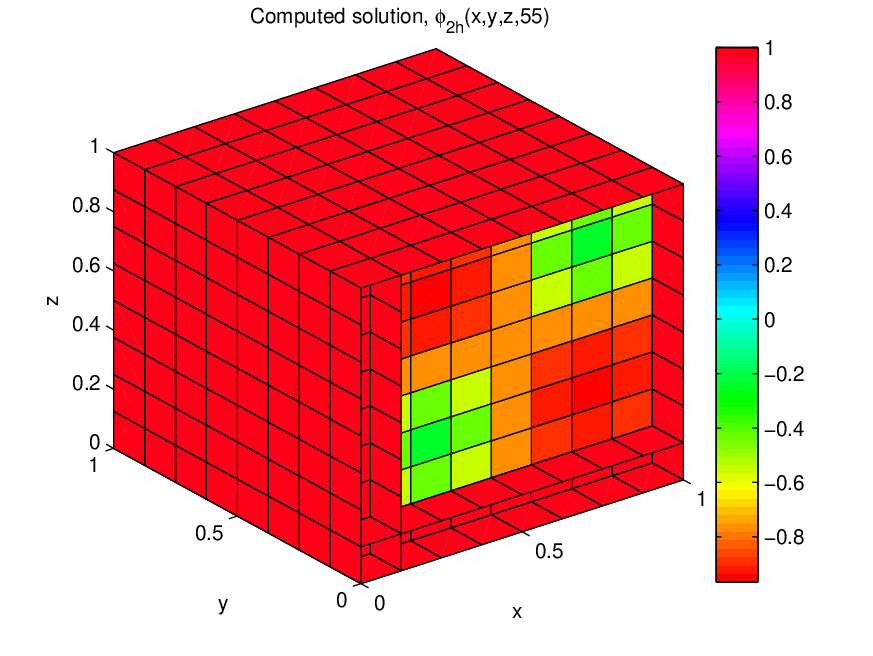,width=7cm}\\
         \psfig{file=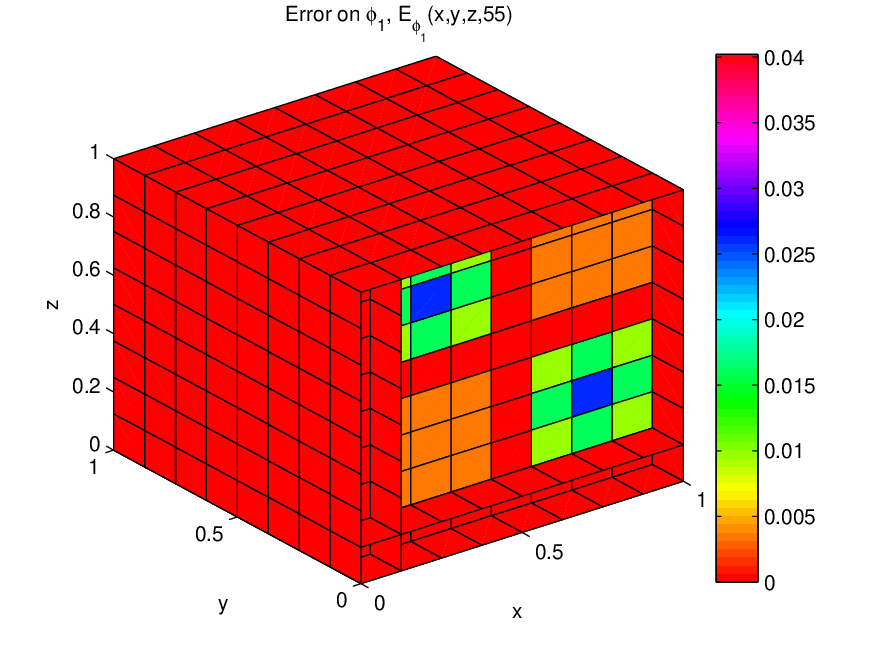,width=7cm} & \psfig{file=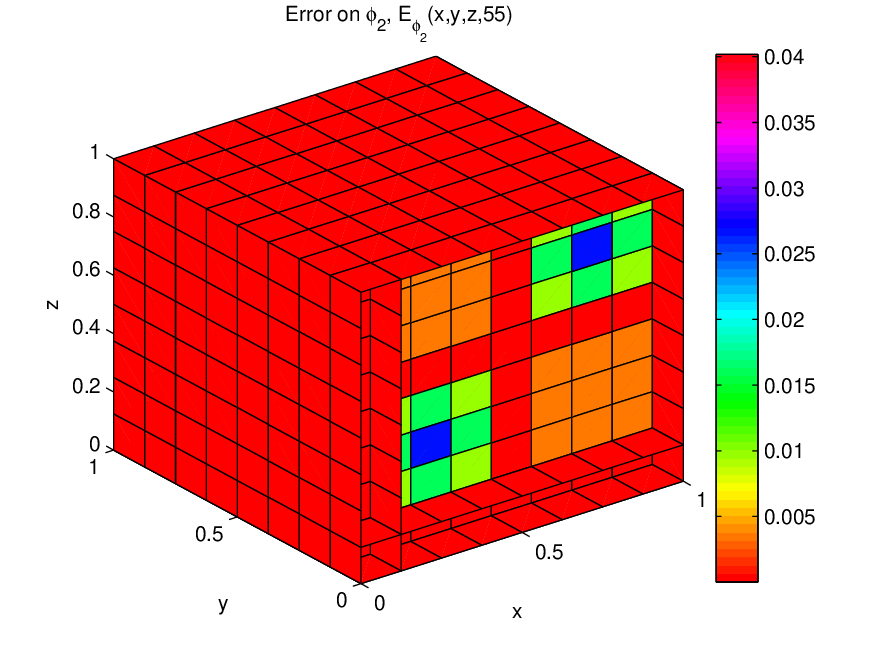,width=7cm}
         \end{tabular}
        \end{center}
         \caption{Graphs of analytical and approximate solutions for Temperature $(T\text{\,and\,}T_{h})$, related approximate radiation intensity
         $(\phi_{1h}\text{\,and\,}\phi_{2h})$ and errors $(E_{T}\text{\,and\,}E_{\phi_{j}})$, corresponding to Example 2.}
          \label{fig2}
          \end{figure}

        \begin{figure}
       \begin{center}
       Stability analysis of the constructed computational technique for nonlinear radiation-conduction model, with $\sigma=10^{-5}$, $h=2^{-3}$.
       \begin{tabular}{c c}
         \psfig{file=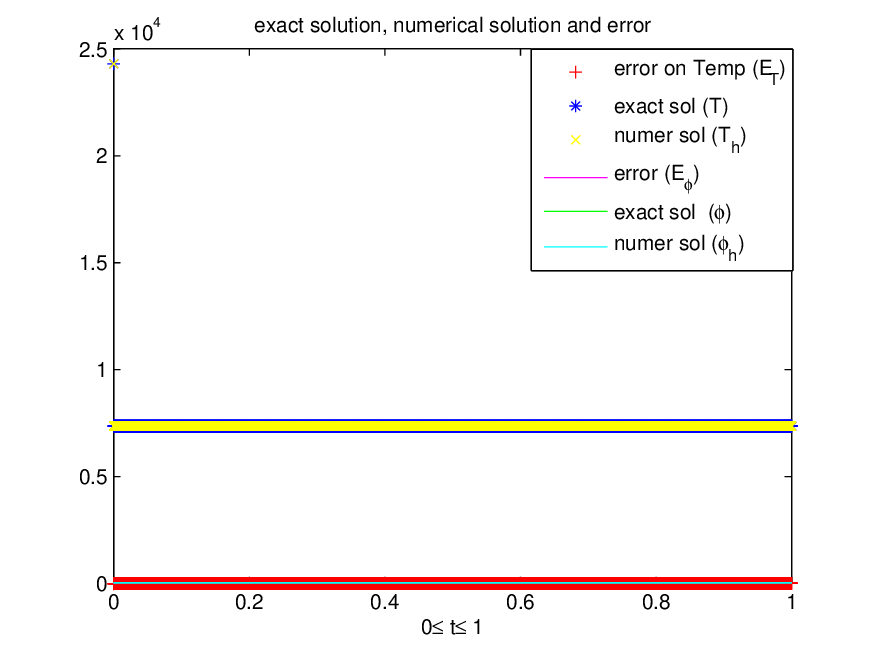,width=7cm} & \psfig{file=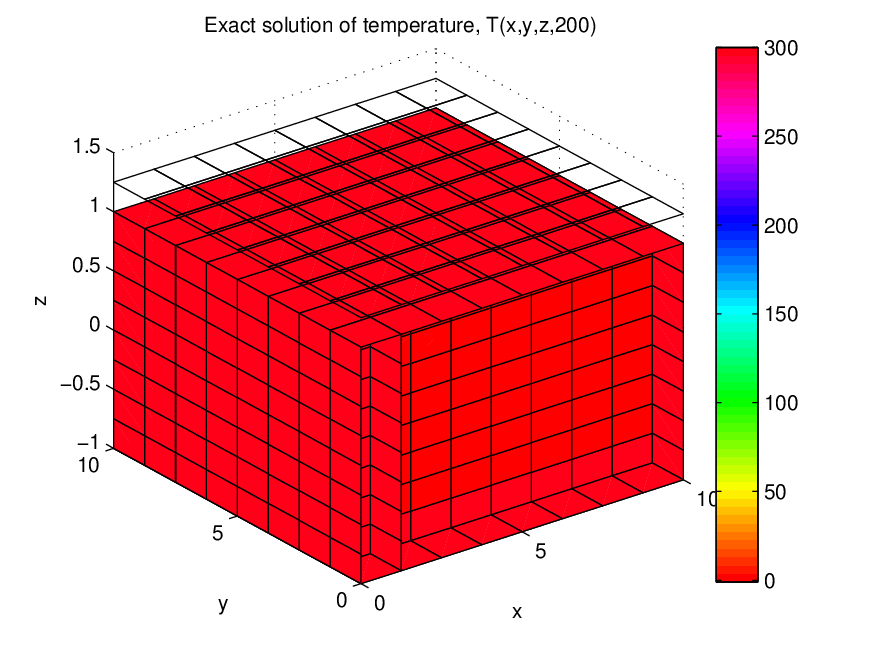,width=7cm}\\
         \psfig{file=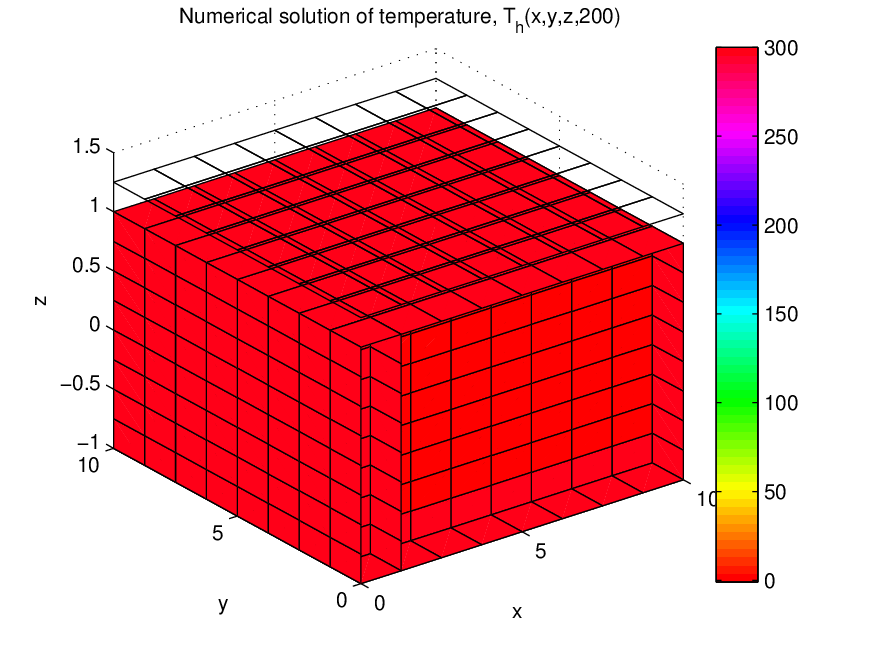,width=7cm} & \psfig{file=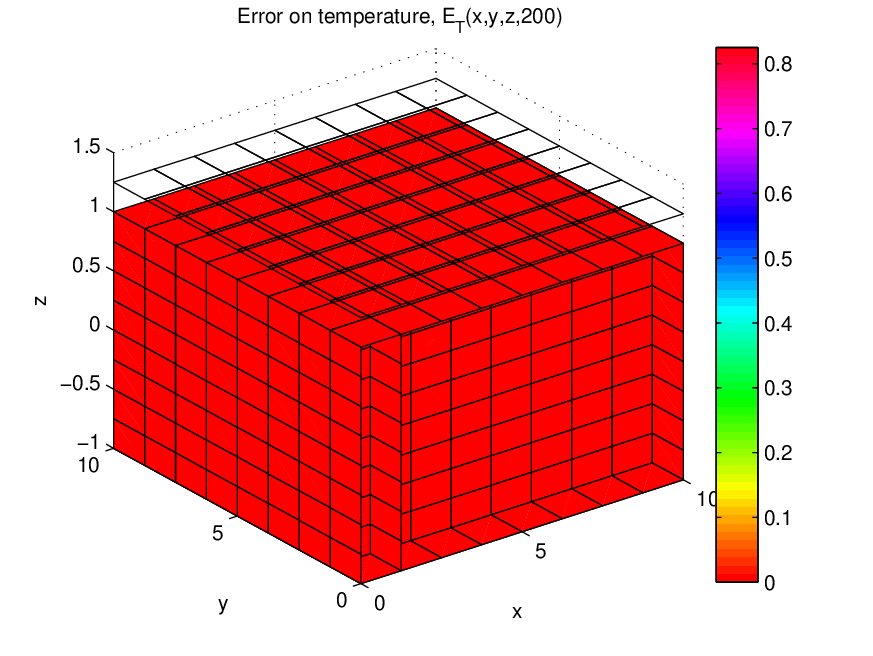,width=7cm}\\
         \psfig{file=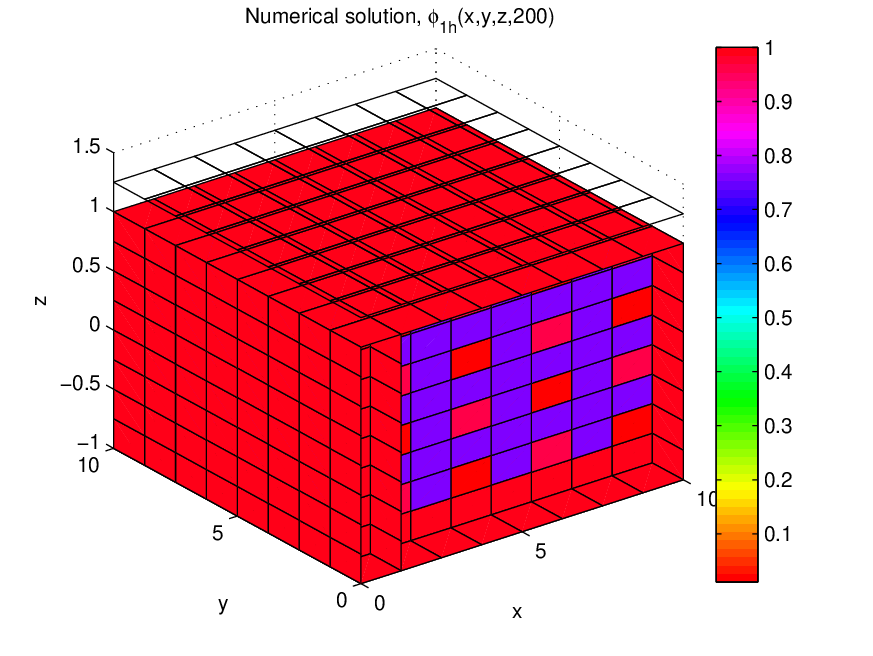,width=7cm} & \psfig{file=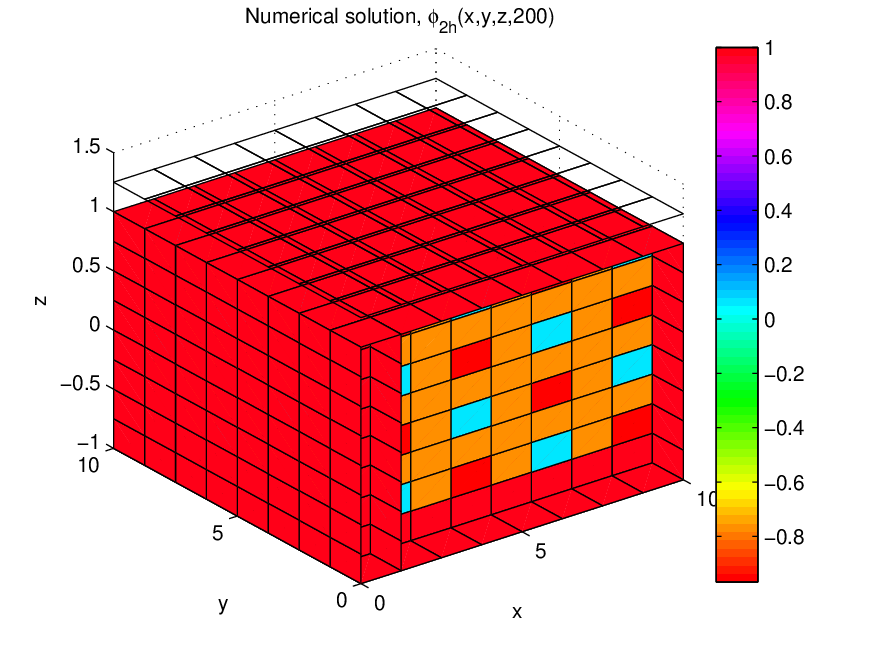,width=7cm}\\
         \psfig{file=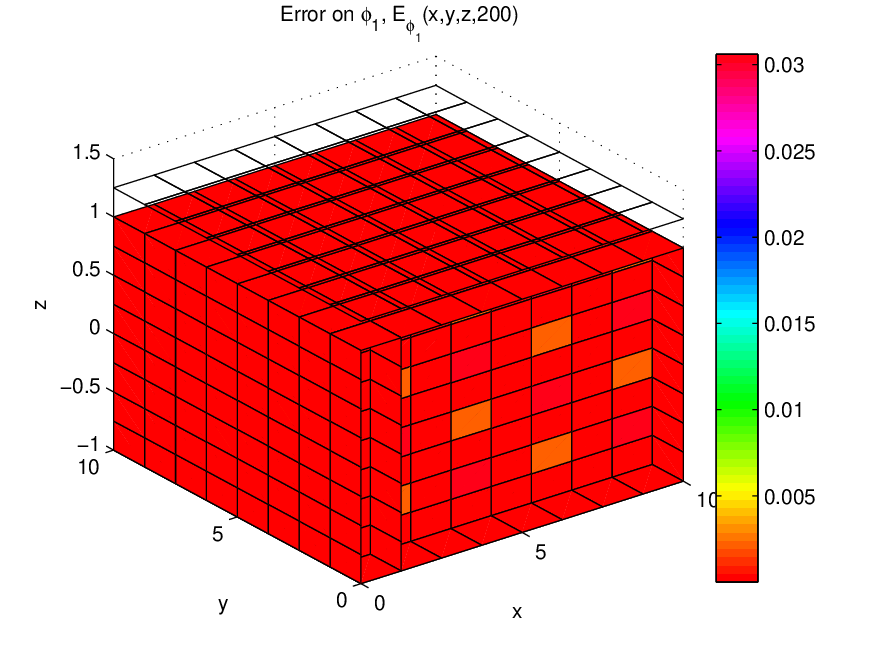,width=7cm} & \psfig{file=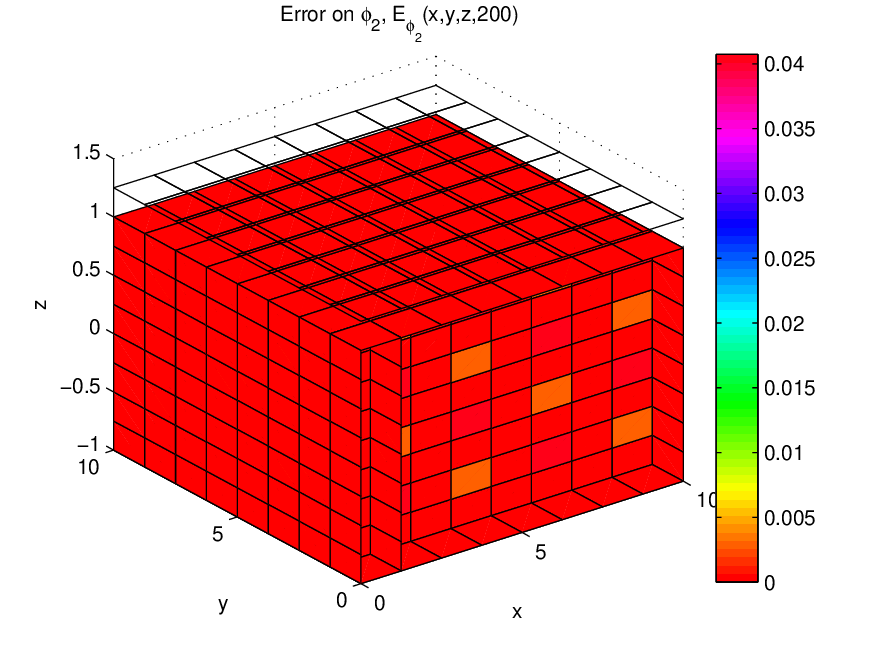,width=7cm}
         \end{tabular}
        \end{center}
         \caption{Graphs of analytical and numerical solutions for Temperature $(T\text{\,and\,}T_{h})$, related approximate radiation intensity
         $(\phi_{1h}\text{\,and\,}\phi_{2h})$ and errors $(E_{T}\text{\,and\,}E_{\phi_{j}})$, corresponding to Example 3.}
          \label{fig3}
          \end{figure}
     \end{document}